\documentclass[11pt,twoside,a4paper]{article}
\usepackage{latexsym}
\usepackage{amsthm,amssymb,amsbsy,amsmath,amsfonts,amssymb,amscd}
\usepackage{graphicx,dsfont}
\usepackage{hyperref}
\usepackage{caption,color}
\usepackage{bbm}
\usepackage{hyperref}
\usepackage{ulem}
\newcommand{\commentout}[1]{}

\newcommand {\Chi} {{\bf \raise 2pt \hbox{$\chi$}} }
\newcommand {\f}   {\frac}

\newcommand{\beq}{\begin{equation}}
\newcommand{\eeq}{\end{equation}}
\newcommand{\bal}{\begin{align}}
\newcommand{\bc}{\begin{cases}}
\newcommand{\ec}{\end{cases}}
\newcommand{\bea} {\begin{array}{rl}}
\newcommand{\eea} {\end{array}}
\newcommand{\bepa}{\left\{ \begin{array}{l}}
\newcommand{\eepa} {\end{array}\right.}
\newtheorem{theorem}{Theorem}[section]
\newtheorem{lemma}[theorem]{Lemma}
\newtheorem{corollary}[theorem]{Corollary}

\newtheorem{prop}[theorem]{Proposition}

\newtheorem{hyp}[theorem]{Assumption}

\def\be{\begin{eqnarray}}
\def\ee{\end{eqnarray}}
\def\ben{\begin{eqnarray*}}
\def\een{\end{eqnarray*}}

\numberwithin{equation}{section}
\numberwithin{figure}{section}

\def\be{\begin{eqnarray}}
\def\ee{\end{eqnarray}}

\def\me{\medskip\noindent}

\newcommand{\Co}{\mathcal{C}}

\def\T{\mathbb{T}}

\def\Z{\mathbb{Z}}
\def\N{\mathbb{N}}
\def\P{\mathbb{P}}
\def\R{\mathbb{R}}
\def\E{\mathbb{E}}

\def\11{\mathbbm{1}}

\title{Filling the gap between individual-based evolutionary models and Hamilton-Jacobi equations}

\author{
 Nicolas Champagnat\thanks{Universit\'e de Lorraine, CNRS, Inria, IECL, F-54000 Nancy, France; E-mail: \texttt{nicolas.champagnat@inria.fr}}
\and {Sylvie M\'el\'eard}\thanks{Ecole Polytechnique et Institut Universitaire de France, CNRS, Institut polytechnique de Paris, route de Saclay, 91128 Palaiseau Cedex-France; E-mail: \texttt{sylvie.meleard@polytechnique.edu}}  \and Sepideh Mirrahimi\thanks{ Institut Montpelli\'erain Alexander Grothendieck, Univ. Montpellier, CNRS, Montpellier, France; E-mail: \texttt{sepideh.mirrahimi@umontpellier.fr}} \and Viet Chi Tran \thanks{LAMA, Univ Gustave Eiffel, Univ Paris Est Creteil, CNRS, F-77454 Marne-la-Vall\'ee, France; E-mail: \texttt{chi.tran@univ-eiffel.fr}}
 }
\date{\today}

\begin{document}
\maketitle
\pagestyle{plain}
\pagenumbering{arabic}

\begin{abstract}
We consider a stochastic model for the evolution of a discrete population structured by a trait with values on a finite grid of the torus, and with mutation and selection. Traits are vertically inherited unless a mutation occurs, and influence the birth and death rates. We focus on a parameter scaling where population is large, individual mutations are small but not rare, and the grid mesh for the trait values is much smaller than the size of mutation steps. When considering the evolution of the population in a long time scale, the contribution of small sub-populations may strongly influence the dynamics. Our main result quantifies the asymptotic dynamics of sub-population sizes on a logarithmic scale. We establish that under the parameter scaling the logarithm of the stochastic population size process, conveniently normalized, converges to the unique viscosity solution of a Hamilton-Jacobi equation. Such   Hamilton-Jacobi equations have already been derived from parabolic integro-differential equations and have been widely developed in the study of adaptation of quantitative traits. Our work provides a justification of this framework directly from a stochastic individual based model, leading to a better understanding of the results obtained within this approach.
The proof makes use of almost sure maximum principles and careful controls of the martingale parts.
\end{abstract}

\me Keywords: stochastic birth death models, large population approximation, selection, mutation, viscosity solution, maximum principle.

\bigskip


\section{Introduction and presentation of the model}
\label{sec:introduction}

Long-term evolutionary dynamics of biological populations may be strongly influenced by small populations and local extinction in some areas of the phenotypical trait space. Survival of small populations in very large populations is crucial when evolution proceeds by selective sweeps \cite{keelingpalmer2008} or for the evolution of antibiotic resistance for bacteria \cite{ochmanlawrencegroisman,stewartlevin}. For example, in bacterial populations involving horizontal transfer, it was shown in~\cite{billiardcolletferrieremeleardtran,calvez20}  that the individual-based long-term dynamics is very sensitive to random survival of very small populations, which may either drive the population to evolutionary suicide or to cyclic dynamics. In such context, the bacterial population is very large making the tracking of small populations very challenging.

From a point of view of mathematical modeling, one wishes to consider large population scalings allowing for survival of much smaller populations. Two approaches emerged: a purely deterministic one, based on partial differential equations (PDE), and a stochastic one, based on birth and death processes (so-called individual-based models in biology). Both approaches describes exponentially small populations sizes and characterize the dynamics of the exponents.

An analytical approach allowing to deal with negligible but non-extinct populations was proposed in~\cite{diekmannjabinmischlerperthame} {and then widely developed (see for instance~\cite{BP.GB:08,Barles2009,LMP}) for the asymptotic study of parabolic integro-differential selection-mutation models}. Let us present it in a setting close to~\cite{Barles2009}. We consider a population whose individuals are differentiated by a trait $x\in \mathbb{T}$, the torus of dimension 1, identified below with the interval $[0,1)$. The trait can vary from an individual to the other.
The evolution of the population is driven by two effects: mutation of the traits, and selection as the reproductive and survival abilities of an individual depend on its trait $x$.
For an individual of trait $x\in \mathbb{T}$, let us denote by $b(x)$ (resp. $d(x)$ and $p(x)$) the clonal birth rate (resp. the death rate and the birth rate with mutation), and by $G(h)$ the mutation kernel. Assuming that the population density solves the PDE
\begin{equation}
  \label{eq:PDE}
\begin{cases}  \varepsilon\partial_t u_\varepsilon(t,x)=u_\varepsilon(t,x)\left(b(x)-d(x)\right)+\int_{\mathbb{T}}\frac{1}{\varepsilon}G\left(\frac{x-y}{\varepsilon}\right)
  p(y)u_\varepsilon(t,y)\mathrm dy & (t,x)\in\R_+\times\T\\
        u_\varepsilon(t,0)=\exp\left(\frac{\beta_0(x)}{\varepsilon}\right), & x\in\T
    \end{cases}
\end{equation}
in the limit $\varepsilon\to 0$ of small mutations and large time, and applying the Hof-Cole transformation
\begin{equation}
    \label{eq:Hopf-Cole}
    \beta_\varepsilon(t,x)=\varepsilon\log u_\varepsilon(t,x),\quad\text{or}\quad u_\varepsilon(t,x)=\exp\left(\frac{\beta_\varepsilon(t,x)}{\varepsilon}\right),
\end{equation}
it is proved in~\cite{Barles2009} (in a slightly different setting{, considering $x\in \R$ and taking into account a competition term}) that $\beta_\varepsilon$ converges to the unique solution $\beta$ of the Hamilton-Jacobi equation
\begin{equation}
    \label{eq:HJ_intro}
    \begin{cases}
        \frac{\partial}{\partial t} \beta(t,x)=b(x)-d(x)+ p(x) \int_{\R} G(h)e^{h\partial_x \beta(t,x)} dh, & (t,x)\in\R_+\times\T\\
        \beta(0,x)=\beta_0(x), & x\in\T
    \end{cases}
\end{equation}

Scaling limits of individual-based models on a discrete trait space with rare mutations and large population, and allowing to deal with negligible populations and local extinction, were proposed in~\cite{durrettmayberry,Bovier2019,champagnatmeleardtran2021,coquille2021,blathpaultobias}. These references focus on population sizes of the order of $K^\beta$ and characterize the asymptotic dynamics of the exponent $\beta$. In particular, local extinction is possible when the exponent $\beta$ hits 0. The fact that the trait space is discrete allows to describe separately the dynamics of each small sub-population.
However, this makes the detailed description of the
asymptotic dynamics very complicated (see~\cite{champagnatmeleardtran2021,coquille2021}).
Note that mutations are assumed individually rare in these references, but they are more frequent than in the scaling limits of adaptive dynamics (see e.g.~\cite{metzgeritzmeszenajacobsheerwaarden,dieckmannlaw,champagnat06,champagnatmeleard2011}), where negligible populations either fixate or go extinct fast, due to the fact that the populational mutation rate vanishes here.
Scaling limits with non-vanishing populational mutation rates partly solve the criticisms raised by biologists~\cite{WaxmanGavrilets2005} concerning the too slow evolutionary speed in adaptive dynamics, particularly in microorganism populations. Biological criticisms were also raised for the analytical approach~\cite{perthame2010}, because
of the so-called {\it tail problem}: exponentially small populations, which may actually be extinct, can have a strong influence on the future evolutionary dynamics of the population. In particular, evolutionary branching is too fast.
Modification of the Hamilton-Jacobi equation were proposed  in~{\cite{perthame2010,Mirrahimi2012,jabin2012}} to solve this problem, but we believe that an individual-based approach is crucial to provide a more realistic and biologically relevant solution to the tail problem.

The purpose of our work is to provide a stochastic individual-based justification of Hamilton-Jacobi equations. To our knowledge, this is the first proof of this kind in the literature.
We follow an individual-based approach, assuming a continuous trait space with a vanishing discretization step $\delta_K$, where $K$ is a scaling parameter such that the population is of the order of $K^{\widetilde{\beta}^K(t,x)}$, assuming frequent and small mutations. In the individual-based model, individuals with trait $x\in\mathbb{T}$ give birth to a clone at rate $b(x)$, die at rate $d(x)$ or give birth to a mutant at rate $p(x)$. Mutant traits are drawn according to a discretization of the distribution $\log K G(\log K \cdot)$. Mutation steps are of the order of $1/\log K$ and the discretization step $\delta_K$ is assumed much smaller than $1/\log K$. In this first work, we focus on the understanding of the relevant scales allowing to capture the limiting Hamilton-Jacobi dynamics. Thus, we consider a simplified model where the birth rate $b$ is assumed larger than the death rate $d$, making the stochastic process super-critical, and the trait space has no boundary. Generalization is a work in progress.

The proof of our main result makes use of uniform Lipschitz bounds on the finite variation part of $\widetilde{\beta}^K$, obtained using an almost sure maximum principle and careful bounds for the martingale part. The identification of the limit is done by checking that it is almost surely viscosity solution of~\eqref{eq:HJ}. We describe the model and state our main result in Section~\ref{sec:model-main-result}. The proof is divided into two main steps---proof of tightness and identification of the limit---which are detailed in Sections~\ref{sec:main-steps} and~\ref{sec:identif}, respectively.

Note that the Hopf-Cole transformation~\eqref{eq:Hopf-Cole} is reminiscent of large deviations scalings. However, our scaling is more of a law of large numbers type. As far as we know, the large deviations interpretation of the Hamilton-Jacobi equation can be done through a Feynman-Kac interpretation of the PDE~\eqref{eq:PDE}~\cite{Champagnat-Henry-2019}. However, the stochastic process involved in the Feynman-Kac formula does not seem to be directly related to the biological population process, even though some works suggest that it may be related to the ancestral trait process of living individuals~\cite{foriengarnierpatout}.

\section{Model and main result}
\label{sec:model-main-result}
\subsection{The model}

We consider a super-critical stochastic birth-death-mutation model describing an asexual population of individuals characterized by a quantitative phenotypic or genetic  trait $x\in \T$. Starting from a finite population whose initial size is parameterized by $K\in \N$, our goal is to recover, in the limit $K\to +\infty$, an evolutionary dynamics described by the Hamilton-Jacobi partial differential equation \eqref{eq:HJ_intro}. For this, we consider a discretization of the trait space $\T$ with step $\delta_K\to 0$. For the sake of simplicity, we will consider in what follows that $1/\delta_K\in \N$. Then, the population is composed of individuals with traits belonging to the discrete space
\[
{\cal X}_K:=\left\{i\delta_K:i\in \{0, 1, \cdots, {1\over \delta_K}-1\}\right\},
\]
embedded with the torus distance: $\forall x, y \in [0,1)$,
\begin{align*}
\rho(x,y) & =\min\left\{|x'-y'|,\ x'=x\text{ mod }1,\ y'=y\text{ mod }1\right\} \\ & =\min\big(|x-y|,1-|x-y|\big).\end{align*}
It's enough to define $\rho(x,y)$ for $x,y \in \T$ by considering their representative in $[0,1)$.

The number of individuals with trait $i \delta_K$ is  described by the stochastic process $(N_i^K(t), t\ge 0)$. The total population size at time $t$  is then given by
$$
N^K(t)=\sum_{i=0}^{1/\delta_K-1}N_i^K(t).
$$
An individual with trait $x\in {\cal X}_K$
\begin{itemize}
    \item gives birth to a new individual with the same trait $x$ at rate $b(x)$;

\item dies  at rate $d(x)$;

    \item gives birth to a mutant individual with trait $y\in{\cal X}_K$ at rate
    \begin{equation}\label{mutation-rate}
        p(x)\delta_K \log K\, G\left((\overline{y-x})\,\log K\right),
    \end{equation}
    where
    $\overline{z}$ is the unique real number of $[-1/2,1/2)$ that is equal to $z$ modulo 1, i.e. $\overline{z}=z-\lfloor z+1/2\rfloor$, where $\lfloor\cdot\rfloor$ is the integer part function.
\end{itemize}


{In the rest of the paper, the following assumptions are made:}

\begin{hyp}\label{hypmodele}
1. We assume that $b$, $d$  and $p$  are nonnegative  Lipschitz continuous  functions  defined on $\T$, such that for all $x\in\T$,
\begin{equation}\label{hyp:supercritical}
b(x)>d(x) \hbox{ and } p(x)>0.
\end{equation}
This means that the birth-death process for each trait is super-critical. In the sequel, we denote by $\bar{b}$, $\bar{p}$ and $\bar{d}$ the upper bounds of these functions on $\T$, by $\,\underline p>0$ the lower bound of $p$, and by $\|b\|_{\text{Lip}}$, $\|d\|_{\text{Lip}}$ and $\|p\|_{\text{Lip}}$ their Lipschitz norm.\\

\noindent 2. The function $G$ defined on $\mathbb{R}$ is  nonnegative, continuous, satisfies $\int_\R G(y)\,dy=1$ and has finite exponential moments of any order. Moreover, we assume that there exists $R>0$ such that $G$ is nonincreasing on $[R,+\infty)$ and nondecreasing on $(-\infty,-R]$.\\
An example of function $G$ satisfying Assumption 2 is given by the Gaussian kernel  $G(h)=\frac{1}{\sqrt{2\pi}\sigma}e^{-h^2/2\sigma^2}$.\\

\noindent 3. There exists a constant $a_1>0$ such that, for all $K\in \N$ and all $i\in\{0, 1, \cdots, {1\over \delta_K}-1\}$,
\begin{equation}\label{hyp:condinit}
    N^K_i(0)\geq K^{a_1}.
\end{equation}

\noindent 4. There exists $a_2<a_1$ such that
\begin{equation}\label{hyp:deltaK}
    K^{-a_2/4}\ll \delta_K\ll \frac{1}{\log K}\quad \mbox{ as }\quad K\rightarrow +\infty.
\end{equation}
\end{hyp}

Point 4 above implies that
\begin{equation}\label{def:hk}
h_K:=\delta_K \log K\ll 1.
\end{equation}
Therefore, the interpretation of \eqref{mutation-rate} is that the rate at which $x$ gives birth to a mutant individual is close to $ p(x)$. Indeed, for an individual with trait $x_K=i_K \delta_K$ with $i_K=\lfloor x/\delta_K\rfloor$ and $x\in\T$ fixed,
\begin{align}\label{eq:conv_integrale}
\lim_{K\rightarrow +\infty}p(x_K)\sum_{j=0}^{{1\over \delta_K}-1} h_K G\left(\overline{(i_K-j)\delta_K}\, \log K\right) & =\lim_{K\rightarrow +\infty}p(x_K)\sum_{\ell=-\lfloor 1/2\delta_K\rfloor}^{1/\delta_K-1-\lfloor 1/2\delta_K\rfloor} h_K G\left(h_K\ell\right) \notag
\\ & =p(x)\int_{\mathbb{R}}G(y)\,dy=p(x),
\end{align}
where we used Assumption~\ref{hypmodele}.2 to control the tails of the Riemann sum.

Therefore, the individual mutation rate is order 1
and mutation steps are small: conditionally on being a mutant, the trait $y$ of the offspring has the distribution $G$ scaled by a factor $1/\log K$ representing the order of magnitude of the mutation steps.
Note also that \eqref{def:hk} means that the mesh size is much smaller than the mutation step.

\medskip
Our goal is to study the asymptotic behavior (when $K$ tends to infinity) of the population sizes $N_i^K$ when $N_i^K(0)$ is of the order of $K^{\alpha_i}$ for some $\alpha_i>0$. Note that in the case where $p(x)=0\,$ for all $x\in\T$, the process $N^K_i(t)$ is a super-critical one-dimensional branching process, hence $\E( N^K_i(t))=\E(N^K_i(0)) e^{(b(i \delta_K)-d(i\delta_K))t}$. Therefore, if the initial condition is of order $K^{\alpha_i}$, then $\E[ N^K_i(t\log K)]\sim K^{\alpha_i + (b(i \delta_K)-d(i\delta_K))t}$.
This suggests to study
\begin{equation}
    \label{beta}
\beta_i^K(t)=\f{\log(N_i^K(t \log K))}{\log(K)},
\end{equation}
with the convention that $\beta^K_i(t)=0\,$ if $\,N^K_i(t\log K)=0$. So the sub-population of trait $i\delta_K$ at time $t\log K$ has size $N^K_i(t \log K)=K^{\beta^K_i(t)}$.

\medskip
 We make the following assumption on the initial condition $\beta^K_i(0)$.

\begin{hyp}\label{hypconvergence}
\noindent Assume that  there exists a constant $A>0$ such that
\begin{equation}\label{hyp:condlip}
\lim_{K\rightarrow+\infty}\mathbb{P}\Big(\sup_{i\neq j}\frac{|\beta^K_i(0)-\beta^K_j(0)|}{\rho(i\, \delta_K,j\, \delta_K)}> A\Big)=0.
\end{equation}
\end{hyp}

\bigskip
{\bf Notations:}
 (i) We shall use the following Riemann approximation repeatedly in the proofs: for all $\alpha>0$,
\begin{multline}\label{eq:momentexpo:G}
\lim_{K\rightarrow +\infty}\sum_{j=0}^{{1\over \delta_K}-1} h_K G\left(\overline{(i_K-j)\delta_K}\,\log K\right) e^{ \alpha\,\log K\,  \rho(j \delta_K,i_K \delta_K)}\\
 =\lim_{K\rightarrow +\infty} \sum_{\ell=-\lfloor 1/2\delta_K\rfloor}^{1/\delta_K-1-\lfloor 1/2\delta_K\rfloor} h_K G\left(h_K\ell\right)e^{ \alpha\,{h_K} \,  | {\ell} |}
=\int_{\mathbb{R}}e^{\alpha |y|}G(y)\,dy,
\end{multline}
and thus, there exists a constant $\overline{G}(\alpha)$ depending only on $\alpha>0$ such that 
\begin{equation}\label{eq:conv_moment-expo}
\sup_{K\geq 1}\,\sum_{j=0}^{{1\over \delta_K}-1} h_K G\left(\overline{(i_K-j)\delta_K}\,\log K\right)e^{ L\,\log K\, \rho(j \delta_K,i_K \delta_K)} =\ :\ \overline G(\alpha)<+\infty.
\end{equation}

\smallskip\noindent (ii) In  what follows and for any function $f$ on $ \{0, 1, \cdots, {1\over \delta_K}-1\}$, we will use the notation
 $$\Delta_Kf_i=  \frac{f_{i+1}-f_i}{\delta_K},$$
 with the convention that $f_{1/\delta_K}=f_0$.

\subsection{The main result - Sketch of the proof}

\par Since we are interested in the convergence of the quantities $\beta_i^K$  to a continuous function defined on the trait space $\T$, when  $K\to+\infty$ and $\delta_K \to 0$, we introduce the following interpolation of the $\beta_i^K$'s:
for all $x\in\T$ and $K\geq 1$, let $i$ be such that $x\in [i\delta_K,(i+1)\delta_K)$, and define
\begin{equation}
\label{lafonction}
 \widetilde \beta^K_{t}(x) =\beta_i^K(t)\Big(1-\f {x}{\delta_K}+i\Big)+\beta_{i+1}^K(t)\Big(\f {x}{\delta_K}-i\Big),
\end{equation}
with the convention that $\beta^K_{1/\delta_K}(t)=\beta^K_0(t)$.
The sequence of processes  $(\widetilde \beta^K_{t}, t\in[0,T])_{K}$  belongs to $\mathbb{D}([0,T], \Co(\T,\mathbb{R}))$, where $\Co(\T,\mathbb{R})$ is endowed with the topology of uniform convergence and $\mathbb{D}([0,T], \Co(\T,\mathbb{R}))$ is the Skorohod space of c\`adl\`ag paths with the associated Skorokhod topology.



Let us state our main theorem.

\begin{theorem}\label{thm:main}
Let $T>0$. Under the Assumptions \ref{hypmodele} and \ref{hypconvergence}, and assuming that $\widetilde{\beta}^K_0(\cdot)$ converges in probability for the topology of uniform convergence on $\Co(\T,\mathbb{R})$ to a deterministic function $\beta_0(\cdot)\in \Co(\T,\mathbb{R})$,
the sequence  $(\widetilde{\beta}^K)_K$ converges in probability in $\mathbb{D}([0,T], \Co(\T,\mathbb{R}))$ to the unique Lipschitz viscosity solution of the Hamilton-Jacobi equation
\begin{equation}
    \label{eq:HJ}
    \begin{cases}
        \frac{\partial}{\partial t} \beta(t,x)=b(x)-d(x)+ p(x) \int_{\R} G(h)e^{h\partial_x \beta(t,x)} dh, & (t,x)\in\R_+\times \T \\
        \beta(0,x)=\beta_0(x), & x\in \T.
     \end{cases}
\end{equation}
\end{theorem}

The proof of Theorem \ref{thm:main} will be classically obtained in two steps: tightness and identification of the limiting values. Therefore we will prove the two next results, respectively in Sections~\ref{sec:main-steps} and~\ref{sec:identif}.

\begin{theorem}
\label{thm:tight}
The sequence of laws of  $(\tilde \beta^K_{t}, t\in[0,T])_{K}$ is relatively compact in ${\cal P}(\mathbb{D}([0,T], \Co(\T,\mathbb{R})))$.
\end{theorem}

\begin{theorem}
\label{thm:identification}
The limiting values $\beta$ of  $(\tilde \beta^K_{t}, t\in[0,T])_{K}$ are characterized as {the} unique viscosity solution of the Hamilton-Jacobi equation
$$
\frac{\partial}{\partial t} \beta(t,x)= b(x)-d(x) + p(x) \int_{\R} G(h)e^{h\partial_x \beta(t,x)} dh.
$$
\end{theorem}

The proofs of these two results require to control the increments of the functions $\widetilde \beta^K_{t}(.)$.
This functions
can also be written as $$\widetilde
\beta^K_{t}(x) 
=(x-i\delta_K)\Delta_K\beta^K_i(t)+\beta^K_i(t).$$
From this expression, we observe that two technical steps are required: to estimate uniformly $\beta_{i}^K(t)$ and to control uniformly $\ \Delta_K\beta^K_i(t)=\frac{\beta_{i+1}^K(t)-\beta_{i}^K(t)}{\delta_K}\ $ (the second estimate being the harder part, it is a major difficulty and constitutes the technical interest of the paper). Such estimates are also obtained in the deterministic derivation of Hamilton-Jacobi equations of type \eqref{eq:HJ} from parabolic integro-differential equations using the maximum principle and the Bernstein method which consists in applying again the maximum principle to the equation satisfied by the increments (see \cite{Barles2009}). Here, since we have stochastic processes we cannot apply the Bernstein method directly. Using the Doob-Meyer decomposition, the stochastic processes can be separated into a finite variation part and a martingale part. We show indeed that the martingale part remains small with our rescaling and we apply the maximum principle almost surely on the finite variation part.

\medskip
Let us detail now the semimartingale Doob-Meyer decomposition of the processes $ \beta_{i}^K,\ i=0,\cdots, 1/\delta_{K}-1$, which can easily be deduced from the Doob-Meyer decomposition of the semi-martingales $N^K_{i},\ i=0,\cdots, 1/\delta_{K}-1$.

\smallskip We have
\begin{equation}\label{decomposition}
\beta_{i}^K(t)=M^K_{i}(t) + A^K_{i}(t)
\end{equation}
with
\begin{align} A_{i}^K(t)&= \beta_{i}^K(0) + {1\over \log K}\int_{0}^{t\log K}\Big(b(i  \delta_{K})N^K_{i}(s) \log\left(1+{1\over N^K_{i}(s)}\right) \label{def:A}\\
& \hspace{5cm}+  d(i  \delta_{K})N^K_{i}(s) \log\left(1-{1\over N^K_{i}(s)}\right)\Big) ds \nonumber\\
& +{1\over \log K} \sum_{\ell=-\lfloor 1/2\delta_K\rfloor}^{1/\delta_K-1-\lfloor 1/2\delta_K\rfloor}
 h_{K}\, p((i+\ell)  \delta_{K}) G(h_{K}\ell)\int_{0}^{t\log K}N^K_{i+\ell}(s)\log\left(1+{1\over N^K_{i}(s)}\right) ds ,\nonumber
\end{align}
with the conventions that, when the index $j\notin \{0,\dots, 1/\delta_K-1\}$,
\begin{align}
    \label{eq:convention}
     & N^K_{j}=N^K_{j-\lfloor j\delta_K\rfloor/
\delta_K}\quad\text{and}\quad p(j\delta_K)=p((j-\lfloor j\delta_K\rfloor/
\delta_K)\delta_K),\qquad \mbox{ when }j\geq 1/\delta_K,\\
& N^K_{j}=N^K_{j+\lceil | j|\delta_K\rceil/
\delta_K}\quad\text{and}\quad p(j\delta_K)=p((j+\lceil |j|\delta_K\rceil/
\delta_K)\delta_K),\qquad \mbox{ when }j<0.
\end{align}
The  process $M^K_{i}$ is a local martingale with quadratic variation
\begin{align}
\langle M^K_{i}\rangle_{t}& =  {1\over \log^2 K}\int_{0}^{t\log K}\Big(b(i  \delta_{K})N^K_{i}(s)\log^2\Big(1+{1\over N^K_{i}(s)}\Big) \label{def:M}\\
& \qquad\qquad\qquad + d(i  \delta_{K})N^K_{i}(s)\log^2\Big(1-{1\over N^K_{i}(s)}\Big)\Big) ds \nonumber\\
& +{1\over \log^2 K}\sum_{\ell=-\lfloor 1/2\delta_K\rfloor}^{1/\delta_K-1-\lfloor 1/2\delta_K\rfloor}
 h_{K} p((i+\ell)  \delta_{K}) G(h_{K}\ell)\int_{0}^{t\log K}N^K_{i+\ell}(s)\log^2\left(1+{1\over N^K_{i}(s)}\right) ds.\nonumber
\end{align}

\bigskip




In Section \ref{sec:main-steps}, we will prove  technical uniform estimates on the martingale part $\Delta_K M^K_{i}(t)$ and the finite variation part $\Delta_K A^K_{i}(t)$ of the processes $\Delta_K\beta^K_{i}(t)$. In that aim, let us introduce  sequences of stopping times playing an important role in the proofs.

Let $a\in(a_2,a_1)$ be fixed during the rest of the proof, $a_1$ being defined in \eqref{hyp:condinit} and $a_2$ in\eqref{hyp:deltaK}. For any $K$,  we define
\begin{equation}
\label{theta}
    \tau'_{K} = \inf\Big\{t\ge 0, \exists i \in \{0, 1, \cdots, {1\over \delta_K}-1\} ; N^K_i(t\log K) < K^a\Big\}.
\end{equation}
Recall Assumption~\ref{hyp:condinit} on the initial condition. It implies that $N^K_{i}(t\log K)\geq 1$ for all $t\le\tau'_K$.

For all  $L>0$, we also define
\begin{equation} \label{def:tau}
\tau_K=\tau_K(L)=\inf\left\{t\geq 0:\exists i\in \{0, 1, \cdots, {1\over \delta_K}-1\},\,|\beta^K_{i+1}(t)-\beta^K_i(t)|>L\delta_K \right\},
\end{equation}with the usual convention that $\beta^K_{1/\delta_K}=\beta^K_0$.
It is easy to check that
\begin{equation*} 
\tau_K(L)=\inf\left\{t\geq 0:\exists i,j\in \{0, 1, \cdots, {1\over \delta_K}-1\},\,|\beta^K_{i}(t)-\beta^K_j(t)|>L\rho(i\delta_K,j\delta_K) \right\}
\end{equation*}
and
\begin{equation} \label{def:tau2}
\tau_K(L)=\inf\left\{t\geq 0:\exists x,y\in \T,\,|\widetilde{\beta}^K_{t}(x)-\widetilde{\beta}^K_t(y)|>L\ \rho(x,y) \right\}.
\end{equation}

We will study the processes until the stopping time
 \begin{equation}
  \theta_K(L)=\tau_K(L)\wedge \tau'_{K}.\label{def:thetaK(L)}
  \end{equation}
Before the stopping time $\theta_K(L)$, the functions $\widetilde{\beta}^K_t$ are Lipschitz and the population size of each trait is controlled by $K^a$, by definition. For each $L$ fixed, we will  provide uniform estimates on the martingale parts $M^K_{i}(t)$ and $\Delta_K M^K_{i}(t)$ and the finite variation parts $A^K_{i}(t)$ and $\Delta_K A^K_{i}(t)$ of the processes $\beta^K_{i}(t)$ and  $\Delta_K\beta^K_{i}(t)$, before the stopping time $\theta_K(L)$. This will allow us to prove the next lemma:
\begin{lemma}
\label{lem:theta}
For all $T>0$, there exists $L_0>0$ large enough, such that
 \[
 \lim_{K\rightarrow+\infty}\mathbb{P}(\theta_K(L_0)>T)=1.
 \]
\end{lemma}

An expression for $L_0$ will be given in \eqref{eq:L0} in the proof.

\section{Proof of the  tightness }
\label{sec:main-steps}

We will use the criterion of Theorem 3.1 in Jakubowski \cite{jakubowski}: let us consider the set $\mathbb{F}$ of  functions $F_{f}$, for $f\in \Co(\T,\mathbb{R})$, defined on $\Co(\T,\mathbb{R})$ by
$$\forall g\in \Co(\T,\mathbb{R}),\  F_{f}(g) =\int_{\T} f(x)g(x)dx.$$
We have to prove  that $\mathbb{F}$ satisfies the required properties:

(i) For each $\varepsilon>0$, there exists a compact set $C_{\varepsilon}\subset \Co(\T,\mathbb{R})$ such that
$$\forall K,\ \mathbb{P}\left(\widetilde \beta^K\in \mathbb{D}([0,T], C_{\varepsilon})\right) >1-\varepsilon.$$
(ii)  For each $f\in \Co(\T,\mathbb{R})$, the sequence of laws of  real-valued processes
\begin{equation}\label{def-Xf}X^K_{f}(\cdot) = \int_\T \widetilde \beta^K(.,x) f(x) dx\end{equation}
is 
tight.

\medskip
Point (i) is the hard part of the proof. Using Ascoli's characterization of compact subsets of $\Co(\T,\mathbb{R})$, we need to obtain estimates related to equi-boundedness and to equi-continuity for the processes $\widetilde \beta^K_{t}(.)$.
The proof relies on Lipschitz estimates (in $x$) of the functions $\widetilde{\beta}^K_t$. In Section \ref{sec:control-martingale}, we show that the martingale part of $\beta^K_i$ remains small with our rescaling and in Section \ref{sec:control-fv}, we apply the maximum principle almost surely on the finite variation part of $\beta^K_i$. This allows us to prove Lemma \ref{lem:theta} in Section \ref{sec:control-incr}. The proof of the tightness is ended in Section \ref{sec:tight}.

\subsection{Control of the martingale part}\label{sec:control-martingale}

Our first estimate will be useful to prove the tightness of the laws of $\widetilde{\beta}^K$. In the sequel, $C$ denotes a constant not depending on any parameter and that may change from line to line.

The following estimate will be used repeatedly: by~\eqref{def:tau2}, for all $t\leq\tau_K(L)$ and all $i,j\leq 1/\delta_K-1$,
\begin{equation}
\label{eq:borne-N_i/N_j}
\frac{N^K_j(t\log K)}{N^K_{i}(t\log K)} = \exp(\log K (\beta_{j}(t)-\beta_{i}(t)) \le  e^{ L\, \rho(j \delta_K,i \delta_K)\, \log K }.
\end{equation}

\begin{lemma}
\label{lem:M}
For all $T>0$, there exists a constant $C$ independent of $K$, $T$, $L$ and $i$ such that, almost surely, for all $t\leq T$ and all $i\in\{0,\ldots,1/\delta_K-1\}$,
\begin{equation}
\label{eq:lem-M-1}
    \langle M^K_i\rangle_{t\wedge\theta_K(L)}\leq \left(\frac{C\overline{G}(L)}{K^a\log K}\right)t.
\end{equation}
In particular,
\begin{equation}
\label{eq:lem-M-2}
    \mathbb{E}\Big(\sup_{t\le T\wedge \theta_{K}(L)}\sup_{i}\langle M^K_{i}\rangle_{t}\Big) \le  {C\overline{G}(L)T \over K^a\log K}
\end{equation}
and for all $A>0$,
\begin{equation}
    \label{maison}
    \mathbb{P}\Big(\sup_{t\le T\wedge \theta_{K}(L)}\sup_{i}|M^K_i(t)|\geq A\Big)\leq\frac{C\overline{G}(L)T}{A^2\delta_K K^a\log K}.
\end{equation}
\end{lemma}

\begin{proof}
It follows from~\eqref{def:M} that
\begin{multline*}
\langle M^K_i\rangle_{t\wedge\theta_K}\leq \frac{C(\bar b+\bar d)}{\log^2 K}\int_0^{(t\wedge \theta_K)\log K}\frac{ds}{N^K_i(s)}\\
+\frac{\bar p}{\log^2 K}\int_0^{(t\wedge\theta_K)\log K}
\sum_{\ell=-\lfloor 1/2\delta_K\rfloor}^{1/\delta_K-1-\lfloor 1/2\delta_K\rfloor}
 h_{K} G(h_{K}\ell)\frac{N^K_{i+\ell}(s)}{(N^K_{i}(s))^2}\,ds.
\end{multline*}
Therefore, using~\eqref{theta} and~\eqref{eq:borne-N_i/N_j},~\eqref{eq:lem-M-1} follows. Then
\begin{align*}
    \mathbb{P}\Big(\sup_{t\le T\wedge \theta_{K}}\sup_{i}|M^K_i(t)|\geq A\Big) & \leq\sum_{i=0}^{1/\delta_K-1}\mathbb{P}\Big(\sup_{t\le T\wedge \theta_{K}}|M^K_i(t)|\geq A\Big) \\ \leq & \frac{1}{A^2}\sum_{i=0}^{1/\delta_K-1}\mathbb{E}\Big(\sup_{t\le T\wedge \theta_{K}}|M^K_i(t)|^2\Big),
\end{align*}
so we obtain~\eqref{maison} using Doob's inequality and~\eqref{eq:lem-M-1}.
\end{proof}


The second step is a technical lemma that we state now.

\begin{lemma}
\label{lem:mg-max-ineq}
Let $t\le T$. Then, for any $\varepsilon>0$ and any $i\in\{0,\ldots,1/\delta_K-1\}$
\begin{equation}
\label{eq:mg-max-ineq}
    \mathbb{P}\Big(\sup_{s\le t\wedge \theta_K(L) } |\Delta_KM^K_{i}(s)|>\varepsilon\Big)\leq  {C\over \varepsilon} \sqrt{ \frac{ \overline{G}(L) t}{\delta_K^2 K^{a}\log K} },
\end{equation}
where the constant $\overline{G}(L)$ is defined in~\eqref{eq:conv_moment-expo}.
\end{lemma}

\begin{proof}
Let $\varepsilon>0$. By the submartingale maximal lemma, we have that \[ \mathbb{P}\Big(\sup_{s\le t} |\Delta_KM^K_{i}(s)|>\varepsilon\Big)\leq  {1\over \varepsilon}\mathbb{E}( |\Delta_KM^K_{i}(t)|) \leq  {1\over \varepsilon} \mathbb{E}\big( |\Delta_KM^K_{i}(t)|^2\big)^{1/2}.\]
Now, Lemma~\ref{lem:M} yields
 \begin{align*}
 \langle\Delta_KM^K_{i} \rangle_{t\wedge \theta_K}
 \leq \frac{2}{\delta_K^2}(\langle M^K_{i+1} \rangle_{t\wedge \theta_K}  + \langle M^K_{i} \rangle_{t\wedge \theta_K} )
    \leq \frac{ 4 C\overline{G}(L) t}{\delta_K^2 K^{a}\log K}.
  \end{align*}
%
 Hence~\eqref{eq:mg-max-ineq} and thus 
 Lemma~\ref{lem:mg-max-ineq} are proved.\end{proof}

 \medskip
Using a similar argument as in the proof of Lemma~\ref{lem:M}, we have the following result.

\begin{corollary}\label{corol:accrM}
 Let $T>0$ and $\varepsilon_{K}= \delta_K^{-1}( K^{a}\log K)^{-1/4}$. We define the event
 \[
 \Omega_{K}(L)=\left\{\sup_{0\leq i\leq 1/\delta_K-1,\ t\le T\wedge \theta_K(L)}  |\Delta_KM^K_{i}(t)| \le \varepsilon_{K}\right\}.
 \]
 Then there exists a constant $C>0$ such that
 $$\mathbb{P}(\Omega_K^c(L))
\leq \frac{C\sqrt{\overline{G}(L)T}}{\delta_K\varepsilon_K \delta_K \sqrt{K^{a}\log K}}= \frac{C\sqrt{\overline{G}(L)T}}{\delta_K (K^a
\log K)^{1/4}} .$$
 \end{corollary}
\medskip

Note that, by \eqref{hyp:deltaK},  $\delta_K^4 K^{a}\log K$ tends to infinity, so $ \mathbb{P}(\Omega_{K}(L))$ tends to $1$, as $K$ goes to infinity.
In addition, since $|j-i|\delta_K \le 1$ for all $0\leq i,j\leq 1/\delta_K-1$, we also have that
\[
\Omega_K(L)\subset\left\{\sup_{0\leq i,j\leq 1/\delta_K-1,\ t\le T\wedge \theta_K(L)}\left|M^K_i(t)-M^K_j(t)\right|\leq\varepsilon_K\right\},
\]
so it also follows from the last corollary that
\begin{equation}
\label{eq:borne-martingale-simple}
    \mathbb{P}\left(\sup_{0\leq i,j\leq 1/\delta_K-1,\ t\le T\wedge \theta_K(L)}\left|M^K_i(t)-M^K_j(t)\right|>\varepsilon_K\right)\leq C\sqrt{\overline{G}(L)T}
\varepsilon_K.
\end{equation}

 From now on, we will work on the probability subspace $\Omega_{K}(L)$.

 \subsection{Control of the finite variation part}\label{sec:control-fv}


 Let us now focus on the finite variation part $A^K_i$. We will prove that

 \begin{prop}
   \label{est-varpart}
   Let $T>0$. Then,  there exists a constant $C_1$  such that for $K$ large enough, for all $t\leq T$ and all $i\in\{0,\ldots,1/\delta_K-1\}$ the following inequality holds almost surely  on $\Omega_K(L)$:
   \begin{equation}
   \label{max-varpart}
       |A^K_i(t\wedge \theta_K(L)) |\leq \max_{0\leq j\leq 1/\delta_K-1}\beta^K_i(0) + C_1 t.   \end{equation}
 \end{prop}

 \begin{proof}

 We only provide the proof of the upper bound on $A_i^K(t\wedge \theta_K)$. The lower bound can be obtained following similar arguments.

 Let $t$ and $s$ be less than $T$ such that $s<t$. Using that $ \log (1+x)\leq x$ and \eqref{eq:borne-N_i/N_j}, we have
 \begin{multline*}
 A_{i}^K(t\wedge \theta_K) - A_{i}^K(s\wedge \theta_K)\\
     \begin{aligned}
 &= {1\over \log K}\int_{(s\wedge \theta_K) \log K}^{(t\wedge \theta_K)\log K}\Big(b(i  \delta_{K})N^K_{i}(u) \log\left(1+{1\over N^K_{i}(u)}\right) 
 + d(i  \delta_{K})N^K_{i}(u) \log\left(1-{1\over N^K_{i}(u)}\right)\Big) du \nonumber\\
& +{1\over \log K} \sum_{\ell=-\lfloor 1/2\delta_K\rfloor}^{1/\delta_K-1-\lfloor 1/2\delta_K\rfloor} h_{K} p((i+\ell)  \delta_{K}) G(h_{K}\ell)\int_{(s\wedge \theta_K) \log K}^{(t\wedge \theta_K)\log K}N^K_{i+\ell}(u)\log\left(1+{1\over N^K_{i}(u)}\right) du \\
&\leq  C(\bar b+ \bar d) (t- s)+  {1\over \log K} \sum_{\ell=-\lfloor 1/2\delta_K\rfloor}^{1/\delta_K-1-\lfloor 1/2\delta_K\rfloor} h_{K} p((i+\ell)  \delta_{K}) G(h_{K}\ell)\int_{(s\wedge \theta_K) \log K}^{(t\wedge \theta_K)\log K}\frac{N^K_{i+\ell}(u)}{N^K_{i}(u)} 
 du\\
&\leq C(\bar b+ \bar d) (t - s)
\\
& \quad + {1\over \log K} \sum_{\ell=-\lfloor 1/2\delta_K\rfloor}^{1/\delta_K-1-\lfloor 1/2\delta_K\rfloor} h_{K} p((i+\ell)  \delta_{K}) G(h_{K}\ell)\int_{(s\wedge \theta_K) \log K}^{(t\wedge \theta_K)\log K}  \exp(\log K (\beta^K_{i+\ell}(u)-\beta^K_{i}(u)) du.
\end{aligned}
 \end{multline*}
Recall that on $\Omega_K(L)$,
$$\beta^K_{j}(u)-\beta^K_{i}(u) =A_{j}^K(u) - A_{i}^K(u)+ M_{j}^K(u) -M_{i}^K(u) \leq A_{j}^K(u) - A_{i}^K(u)+\varepsilon_K
.$$
Thus we  obtain that, on the event $\Omega_K(L)$,
 \begin{multline*} A_{i}^K(t\wedge\theta_K) - A_{i}^K(s\wedge\theta_K)
 \leq C(\bar b+ \bar d) (t - s)
 + {1\over \log K} \sum_{\ell=-\lfloor 1/2\delta_K\rfloor}^{1/\delta_K-1-\lfloor 1/2\delta_K\rfloor} h_{K} p((i+\ell)  \delta_{K}) G(h_{K}\ell) \\ \int_{(s\wedge \theta_K) \log K}^{(t\wedge \theta_K)\log K} e^{\varepsilon_K \log K } \exp(\log K (A^K_{i+\ell}(u)-A^K_{i}(u))) du.
\end{multline*}
We deduce  that, almost surely on $\Omega_K(L)$ and for all $t\leq\theta_K(L)$,
\begin{multline}
\frac{dA^K_{i}(t)}{dt} \leq C(\bar b+\bar d) \\
+ {1\over \log K} \sum_{\ell=-\lfloor 1/2\delta_K\rfloor}^{1/\delta_K-1-\lfloor 1/2\delta_K\rfloor} h_{K} p((i+\ell)  \delta_{K}) G(h_{K}\ell) e^{\varepsilon_K\log K } \exp(\log K (A^K_{i+\ell}(t)-A^K_{i}(t)))\end{multline}
Defining $\widetilde A^K_i(t) = A^K_{i}(t) - 2 C(\bar b+\bar d)t - 2 \bar p t$, we deduce that for any $t\leq \theta_K$,
\begin{equation}
\label{eq:prop-est-varpart}
    {d \widetilde A^K_{i}(t)\over dt}< \bar p e^{\varepsilon_K\log K }  \sum_{\ell=-\lfloor 1/2\delta_K\rfloor}^{1/\delta_K-1-\lfloor 1/2\delta_K\rfloor} h_{K}   G(h_{K}\ell)  \exp(\log K (\widetilde A^K_{i+\ell}(t)- \widetilde A^K_{i}(t)) - 2 \bar p.
\end{equation}
Let us introduce
  $$(i_{K}, t_{K}) = (i_K(\omega),t_K(\omega))=\text{argmax}_{{i\in \{0,\cdots, {1\over \delta_{K}}-1\}, t\in [0, \theta_{K}(\omega)]}} \widetilde A^K_{i}(t).$$
We can prove that $t_{K}=0.$
Indeed, if conversely we assume that $t_{K}>0$, then  the right term of \eqref{eq:prop-est-varpart}, for $K$ large enough,  is non positive for $i=i_{K}$ {and $t=t_K$} and then the
left term is negative, contradicting the fact that $\widetilde A^K_{i_K}(t)$ is maximal for $t=t_K$. Hence, we have proved that {for $K$ large enough}, almost surely on
the event $\Omega_K$, for all $t\leq \theta_K$ and $0\leq i\leq 1/\delta_K-1$,
\begin{align*}
A^K_i(t)=\widetilde A^K_i(t)+2C(\bar b+\bar d) t  + 2 \bar p t\leq &\max_{0\leq j\leq 1/\delta_K}\widetilde A^K _j(0)+2 C(\bar b+\bar d) t + 2 \bar p t \\
= & \max_{0\leq j\leq 1/\delta_K} \beta^K _j(0)+2C(\bar b+\bar d) t +  2 \bar p t.
\end{align*}
 \end{proof}

 \bigskip

The last result has a consequence that will be useful to prove the tightness of $\widetilde{\beta}^K$ in Section~\ref{sec:tight}.

\begin{corollary}
\label{lem:upper-bound}
For all $T>0$, there exists $C(T)$ such that,
\[
\lim_{K\to+\infty}\mathbb{P}\left(\sup_{0\leq i\leq 1/
\delta_K-1}\sup_{t\in[0,T\wedge\theta_K(L)]}\beta^K_i(t)\geq C(T)\right)=0.
\]
\end{corollary}

\begin{proof}
We use the semimartingale decomposition~\eqref{decomposition} of $\beta^K_i$, Proposition~\ref{est-varpart},~\eqref{maison} with $A=1$ and Corollary~\ref{corol:accrM} to deduce that, for all $t\leq T$ {and $K$ large enough},
\[
|\beta^K_i(t\wedge\theta_K)|\leq \sup_{0\leq j\leq 1/\delta_K-1}|\beta^K_j(0)|+C_1T+1
\]
with probability at least $1-\frac{C\overline{G}(L)T}{\delta_K K^a\log K} -\frac{C\sqrt{\overline{G}(L)T}}{\delta_K(K^a \log K)^{1/4}}$. Since $\widetilde{\beta}^K_i(0)$ converges in probability to $\beta_0$, $\mathbb{P}(\sup_i \beta^K_i(0)\geq \|\beta_0\|_\infty+1)$ converges to $0$ when $K$ goes to $+\infty$. Hence the result follows with $C(T)=\|\beta_0\|_\infty+C_1T+2$.
\end{proof}

 \subsection{Proof of Lemma \ref{lem:theta}}\label{sec:control-incr}

The proof of Lemma \ref{lem:theta} results from the following two results:

\begin{lemma}\label{lem:lower-bound} Under Assumption~\ref{hypmodele}.1 and~\ref{hypmodele}.3,  
for all $T>0$,
\[
\lim_{K\to+\infty}\mathbb{P}(\tau'_K\geq T) =1.
\]
\end{lemma}

 \begin{proof}[Proof of Lemma \ref{lem:lower-bound}]
 By a coupling procedure, it is easy to prove that for each $i$, the process $(N_i^K(t))_{t}$ is  pathwisely bounded below by a branching process  $(Z_i^K(t))_{t}$ with birth rate $b(i\delta_K)$, death rate $d(i\delta_K)$ and initial condition $K^{a+\varepsilon}$ for $\varepsilon=a_1-a>0$.
 In addition, the processes $(Z_i^K(t))_{t}$ for $0\leq i\leq {1\over \delta_K}-1$ are independent. Let us define
\[
\theta''_{K} = \inf\Big\{t\ge 0, \exists i \in \{0, 1, \cdots, {1\over \delta_K}-1\} ; Z^K_i(t\log K) < K^a\Big\} .
\]
 In order to prove that $\ \lim_{K\to \infty} \mathbb{P}(\tau'_{K}>T) =1$, it is enough to prove that
 $$\ \lim_{K\to \infty} \mathbb{P}(\theta''_{K}=+\infty) =1.$$
 We have
 \begin{eqnarray*}
 \mathbb{P}(\theta''_{K}=+\infty) &=& \mathbb{P}\Big(\forall i\in \{0,1, \cdots, {1\over \delta_K}-1\} ,\  \forall t\geq 0\  Z^K_i(t \log K) >K^a\Big)\\
 &=& 
 \prod_{i=0}^{{1\over \delta_K}-1}\mathbb{P}\Big( \inf_{t\ge 0} Z^K_i(t\log K) >K^a\Big).
 \end{eqnarray*}
 Fix $i\in\{0,\ldots,{1\over \delta_K}-1\}$. It is usual to prove (by time change) that the probability $\mathbb{P}\big( \inf_{t\ge 0} Z^K_i(t\log K) >K^a\big)$ is equal to the probability for a random walk $M({b(i\delta_K)\over b(i\delta_K)+d(i\delta_K)}, {d(i\delta_K)\over b(i\delta_K)+d(i\delta_K)})$ on $\Z_{+}$ (adding $+1$ with probability ${b(i\delta_K)\over b(i\delta_K)+d(i\delta_K)}$ and $-1$ with probability ${d(i\delta_K)\over b(i\delta_K)+d(i\delta_K)}$) with initial value $K^{a+\varepsilon}$ never to attain $K^a$.
 This quantity is well known and equal to
 $$1 -\Big( {d(i\delta_K)\over b(i\delta_K)}\Big)^{K^{a+\varepsilon}-K^a}.$$
 Since $\alpha=\max_{x\in\T}d(x)/b(x)<1$, it follows from \eqref{hyp:deltaK} that
 \begin{eqnarray*}
 \mathbb{P}(\theta''_{K}=+\infty) &\geq& \exp\bigg({1\over \delta_K} \log\big(1 -\alpha^{K^{a+\varepsilon}-K^a}\big)\bigg)\\
 &\sim& \exp\bigg( - {1\over \delta_K}\alpha^{K^{a+\varepsilon}-K^a}\bigg)\\
 &\gg&  1-K^{{a_2}/4} \alpha^{K^{a+\varepsilon}-K^a},
 \end{eqnarray*}
which tends to $1$ when $K$ tends to infinity.
 \end{proof}

\begin{prop}
  \label{prop:Lipschitz-estimate}
 Under the Assumptions \ref{hypmodele} and \ref{hypconvergence}, for all $T>0$, there exists $L_0$ in the definition \eqref{def:tau} of $\tau_K(L)$ such that
 \[
 \lim_{K\rightarrow+\infty}\mathbb{P}(\tau_K(L_0)>T)=1.
 \]
\end{prop}

\begin{proof}[Proof of Proposition \ref{prop:Lipschitz-estimate}]

   For $i\in\{0,\ldots,1/\delta_K-1\}$, let us consider the increments
$$
\Delta_K\beta^K_{i}(t\wedge \theta_K)= \frac{\beta_{i+1}^K(t\wedge \theta_K)-\beta_{i}^K(t\wedge \theta_K)}{\delta_K}=\Delta_KM^K_{i}(t\wedge \theta_K) + \Delta_KA^K_{i}(t\wedge \theta_K).
$$
Our aim is to prove  that $\Delta_K\beta^K_{i}(t)-\Delta_K\beta^K_{i}(s)$ is close to its finite variation part for large $K$ and to apply to the latter an almost sure maximum principle.


 Let us introduce
 \begin{equation}
 \label{def:g}
  g^K_{i}(t)=\Delta_K  A^K_{i}(t\wedge \theta_K)+\frac{\|p\|_{\text{Lip}}}{\underline p} A^K_{i+1}(t\wedge\theta_K),
 \end{equation}

Using \eqref{def:A}, we obtain for {$K$ large enough and} $0\leq s<t\leq T$,
  {\small  \begin{align}
 & g^K_{i}(t)-g^K_{i}(s) \notag \\
     & =\frac{1}{h_K}\int_{(s\wedge\theta_K)\log K}^{(t\wedge\theta_K)\log K}\left[b((i+1)\delta_K)N^K_{i+1}(u) \log\left(1+{1\over N^K_{i+1}(u)}\right)-b(i\delta_K)N^K_{i}(u)
      \log\left(1+{1\over N^K_{i}(u)}\right)\right]du \notag \\
      & +\frac{1}{h_K}\int_{(s\wedge\theta_K)\log K}^{(t\wedge\theta_K)\log K} \left[d((i+1)\delta_K)N^K_{i+1}(u) \log\left(1-{1\over N^K_{i+1}(u)}\right)-d(i\delta_K)N^K_{i}(u)
      \log\left(1-{1\over N^K_{i}(u)}\right)\right]du \notag \\
      & +\frac{1}{h_K}\int_{(s\wedge\theta_K)\log K}^{(t\wedge\theta_K)\log K} \sum_{\ell=-\lfloor 1/2\delta_K\rfloor}^{1/\delta_K-1-\lfloor 1/2\delta_K\rfloor} h_{K}  G(h_{K}\ell)
    \notag\\
    &\hskip 1cm \left[ p((i+1+\ell)  \delta_{K})N^K_{i+1+\ell}(u)\log\left(1+{1\over N^K_{i+1}(u)}\right)-p((i+\ell)  \delta_{K})N^K_{i+\ell}(u)\log\left(1+{1\over N^K_{i}(u)}\right))\right]du \notag \\
    & +\frac{\|p\|_{\text{Lip}}}{\underline p\log K} \int_{(s\wedge\theta_K)\log K}^{(t\wedge\theta_K)\log K}\left[b((i+1)\delta_K)N^K_{i+1}(u) \log\left(1+{1\over N^K_{i+1}(u)}\right)\right.\notag\\
    & \hspace{4cm}+\left. d((i+1)\delta_K)N^K_{i+1}(u) \log\left(1-{1\over N^K_{i+1}(u)}\right)\right]du \notag \\
    & +\frac{\|p\|_{\text{Lip}}}{\underline p\log K} \int_{(s\wedge\theta_K)\log K}^{(t\wedge\theta_K)\log K} \sum_{\ell=-\lfloor 1/2\delta_K\rfloor}^{1/\delta_K-1-\lfloor 1/2\delta_K\rfloor} h_{K}  G(h_{K}\ell)
    p((i+1+\ell)\delta_K)N^K_{i+1+\ell}(u)\log\left(1+{1\over N^K_{i+1}(u)}\right)du \notag \\
  &   \leq \frac{C(\bar b+\bar d)}{ h_K}\int_{(s\wedge\theta_K)\log K}^{(t\wedge\theta_K)\log K}\left[\frac{1}{N^K_{i+1}(u)}+\frac{1}{N^K_{i}(u)}\right]du +(C_1(\bar b+\bar d)+\|b\|_{\text{Lip}}+\|d\|_{\text{Lip}})(t\wedge \theta_K-s\wedge \theta_K) \notag \\
    & +\frac{1}{h_K}\int_{(s\wedge\theta_K)\log K}^{(t\wedge\theta_K)\log K} \sum_{\ell=-\lfloor 1/2\delta_K\rfloor}^{1/\delta_K-1-\lfloor 1/2\delta_K\rfloor} h_{K} p((\ell+i)\delta_K) G(h_{K}\ell)\left[N_{\ell+i+1}^K(u)\log \left(1+\frac{1}{N_{i+1}^K(u)}\right)\right.\notag\\
    & \hspace{6cm}\left.-N_{\ell+i}^K(u)\log \left(1+\frac{1}{N_{i}^K(u)}\right)
    \right]du \notag \\
    & +\frac{3\|p\|_{\text{Lip}}}{\underline{p}\log K}\int_{(s\wedge\theta_K)\log K}^{(t\wedge\theta_K)\log K} \sum_{\ell=-\lfloor 1/2\delta_K\rfloor}^{1/\delta_K-1-\lfloor 1/2\delta_K\rfloor}h_Kp((\ell+i)\delta_K)G(h_K\ell)N_{\ell+i+1}^K(u)\log \left(1+\frac{1}{N_{i+1}^K(u)}\right)du, \label{eq:horrible}
  \end{align}}
where we used that for all {$x$ such that $|x|\leq 1/2$  we have
\begin{align}
    \left|\frac{1}{x}\log(1+x)\right|\leq C,
    \label{approxlog3}
\end{align}
\begin{align}
    \left|\frac{1}{x}\log(1+x)-\frac{1}{y}\log(1+y)\right|\leq C(|x|+|y|).
     \label{approxlog4}
\end{align}}
and the fact that (recalling the convention~\eqref{eq:convention} and that $p$ is periodic)
\begin{align}
\left|p((\ell+i+1)\delta_K)-p((\ell+i)\delta_K)\right|=& p((\ell+i)\delta_K)\frac{\left|p((\ell+i+1)\delta_K)-p((\ell+i)\delta_K)\right|}{p((\ell+i)\delta_K)} \notag \\
\leq &\frac{\|p\|_{\text{Lip}}\delta_K}{\underline{p}}p((\ell+i)\delta_K)
\label{eq:borne-Lipschitz-ppe-max}
\end{align}
in the last inequality. {Note also that, to obtain this last inequality we have taken $K$ large enough such that
$$
\frac{\|p\|_{\text{Lip}}\delta_K}{\underline{p}}\leq 1,
$$
so that
$$
p((\ell+i+1)\delta_K)\leq 2p((\ell+i)\delta_K).
$$}
{Next,} notice that for all {$x',y',x,y$ such that $|x|,|y|\leq 1/2$,
\begin{align}
\label{approx-bis}
{1 \over y'}\log(1+{  y}) -{ 1\over x'}\log(1+{ x}) \le {y\over y'} - {x\over x'} + C\Big({y^2\over y'}+{x^2\over x'}\Big).
\end{align} }
Using this inequality and \eqref{eq:borne-N_i/N_j}, we have
\begin{align*}
\frac{1}{h_K} & \int_{(s\wedge\theta_K)\log K}^{(t\wedge\theta_K)\log K} \sum_{\ell=-\lfloor 1/2\delta_K\rfloor}^{1/\delta_K-1-\lfloor 1/2\delta_K\rfloor} h_{K} p((\ell+i)\delta_K) G(h_{K}\ell)\left[N_{\ell+i+1}^K(u)\log \left(1+\frac{1}{N_{i+1}^K(u)}\right)\right.\\
    & \hspace*{6cm} \left.-N_{\ell+i}^K(u)\log \left(1+\frac{1}{N_{i}^K(u)}\right)
    \right]du \\
 & \leq \frac{1}{h_K}\int_{(s\wedge\theta_K)\log K}^{(t\wedge\theta_K)\log K} \sum_{\ell=-\lfloor 1/2\delta_K\rfloor}^{1/\delta_K-1-\lfloor 1/2\delta_K\rfloor} h_Kp((\ell+i)\delta_K)G(h_K\ell)\left[\frac{N_{\ell+i+1}^K(u)}{N_{i+1}^K(u)}-\frac{N_{\ell+i}^K(u)}{N_{i}^K(u)}
    \right]du \\ & +\frac{C\bar p}{h_K}\int_{(s\wedge\theta_K)\log K}^{(t\wedge\theta_K)\log K} \sum_{\ell=-\lfloor 1/2\delta_K\rfloor}^{1/\delta_K-1-\lfloor 1/2\delta_K\rfloor} h_KG(h_K\ell)e^{Lh_K|\ell|}\left(\frac{1}{N_{i+1}^K(u)}+\frac{1}{N_{i}^K(u)}\right)du.
\end{align*}
Therefore, using \eqref{eq:conv_moment-expo} and the definition of $\tau'_K$, we have proved that
\begin{align*}
    g^K_{i}(t)- & g^K_{i}(s)
     \leq \left(\frac{C\bar G(L)\log K}{ K^a \ h_K}+C\right)(t-s) \\
    & +\frac{1}{h_K}\int_{(s\wedge\theta_K)\log K}^{(t\wedge\theta_K)\log K} \sum_{\ell=-\lfloor 1/2\delta_K\rfloor}^{1/\delta_K-1-\lfloor 1/2\delta_K\rfloor} h_Kp((\ell+i)\delta_K)G(h_K\ell)\left[\frac{N_{\ell+i+1}^K(u)}{N_{i+1}^K(u)}-\frac{N_{\ell+i}^K(u)}{N_{i}^K(u)}
    \right]du \\
    & +\frac{{3}\|p\|_{\text{Lip}}}{\underline{p}\log K}\int_{(s\wedge\theta_K)\log K}^{(t\wedge\theta_K)\log K} \sum_{\ell=-\lfloor 1/2\delta_K\rfloor}^{1/\delta_K-1-\lfloor 1/2\delta_K\rfloor} h_Kp((\ell+i)\delta_K)G(h_K\ell)\frac{N_{\ell+i+1}^K(u)}{N_{i+1}^K(u)}du.
\end{align*}
Using that for any real numbers $\lambda,\alpha$, $e^\lambda\le e^\alpha+ e^\lambda(\lambda-\alpha)$ and then~\eqref{decomposition},
we have
\begin{align*}
    & \frac{1}{h_K}\int_{(s\wedge\theta_K)\log K}^{(t\wedge\theta_K)\log K} \sum_{\ell=-\lfloor 1/2\delta_K\rfloor}^{1/\delta_K-1-\lfloor 1/2\delta_K\rfloor} h_Kp((\ell+i)\delta_K)G(h_K\ell)\left[\frac{N_{\ell+i+1}^K(u)}{N_{i+1}^K(u)}-\frac{N_{\ell+i}^K(u)}{N_{i}^K(u)}
    \right]du
    \\
    &   \leq \frac{\log K}{ h_K}\int_{(s\wedge\theta_K)\log K}^{(t\wedge\theta_K)\log K} \sum_{\ell=-\lfloor 1/2\delta_K\rfloor}^{1/\delta_K-1-\lfloor 1/2\delta_K\rfloor} h_Kp((\ell+i)\delta_K)G(h_K\ell) \\ & \qquad\qquad\left(\beta^K_{\ell+i+1}\big(\frac{u}{\log K}\big)-\beta^K_{\ell+i}\big(\frac{u}{\log K}\big)-(\beta^K_{i+1}\big(\frac{u}{\log K}\big)-\beta^K_i\big(\frac{u}{\log K}\big))\right)
  \frac{N^K_{\ell+i+1}(u)}{N^K_{i+1}(u)}du \\
  & \leq \int_{(s\wedge\theta_K)\log K}^{(t\wedge\theta_K)\log K} \sum_{\ell=-\lfloor 1/2\delta_K\rfloor}^{1/\delta_K-1-\lfloor 1/2\delta_K\rfloor} h_Kp((\ell+i)\delta_K)G(h_K\ell)\\  & \qquad\qquad\left(\Delta_K\beta^K_{\ell+i}\big(\frac{u}{\log K}\big)-\Delta_K\beta^K_{i}\big(\frac{u}{\log K}\big)\right)
  \frac{N^K_{\ell+i+1}(u)}{N^K_{i+1}(u)}du \\
  &\leq \int_{(s\wedge\theta_K)\log K}^{(t\wedge\theta_K)\log K} \sum_{\ell=-\lfloor 1/2\delta_K\rfloor}^{1/\delta_K-1-\lfloor 1/2\delta_K\rfloor} h_Kp((\ell+i)\delta_K)G(h_K\ell)\\  & \qquad\left(\Delta_KM^K_{\ell+i}\big(\frac{u}{\log K}\big)-\Delta_KM^K_{i}\big(\frac{u}{\log K}\big)+ \Delta_KA^K_{\ell+i}\big(\frac{u}{\log K}\big)-\Delta_KA^K_{i}\big(\frac{u}{\log K}\big)\right)
  \frac{N^K_{\ell+i+1}(u)}{N^K_{i+1}(u)}du .
\end{align*}
Thus, using~\eqref{hyp:deltaK} and~\eqref{def:g}, we deduce that
{\small
\begin{multline*}
     g^K_{i}(t)-g^K_{i}(s) \\ -\int_{(s\wedge\theta_K)\log K}^{(t\wedge\theta_K)\log K} \sum_{\ell=-\lfloor 1/2\delta_K\rfloor}^{1/\delta_K-1-\lfloor 1/2\delta_K\rfloor}
  h_{K} p((\ell+i)\delta_K)G(h_{K}\ell)\left(g_{\ell+i}\big(\frac{u}{\log K}\big)-g_i\big(\frac{u}{\log K}\big)\right)
  \frac{N^K_{\ell+i+1}(u)}{N^K_{i+1}(u)}du \\
  \begin{aligned}
& \leq C_0(K,L)(t-s)+\frac{{3}\|p\|_{\text{Lip}}}{\underline{p}\log K}\int_{(s\wedge\theta_K)\log K}^{(t\wedge\theta_K)\log K} \sum_{\ell=-\lfloor 1/2\delta_K\rfloor}^{1/\delta_K-1-\lfloor 1/2\delta_K\rfloor}h_Kp((\ell+i)\delta_K)G(h_K\ell)\frac{N_{\ell+i+1}^K(u)}{N_{i+1}^K(u)}du \\
  & -\frac{\|p\|_{\text{Lip}}}{\underline{p}}\int_{(s\wedge\theta_K)\log K}^{(t\wedge\theta_K)\log K} \sum_{\ell=-\lfloor 1/2\delta_K\rfloor}^{1/\delta_K-1-\lfloor 1/2\delta_K\rfloor}
  h_{K} p((\ell+i)\delta_K)G(h_{K}\ell)\\
  & \hspace{7cm}\left(\beta_{\ell+i+1}\big(\frac{u}{\log K}\big)-\beta_{i+1}\big(\frac{u}{\log K}\big)\right)
  \frac{N^K_{\ell+i+1}(u)}{N^K_{i+1}(u)}du \\
  & + \frac{\|p\|_{\text{Lip}}}{\underline{p}}\int_{(s\wedge\theta_K)\log K}^{(t\wedge\theta_K)\log K} \sum_{\ell=-\lfloor 1/2\delta_K\rfloor}^{1/\delta_K-1-\lfloor 1/2\delta_K\rfloor}
  h_{K} p((\ell+i)\delta_K)G(h_{K}\ell)\\
  & \hspace{7cm}\left(M_{\ell+i+1}\big(\frac{u}{\log K}\big)-M_{i+1}\big(\frac{u}{\log K}\big)\right)
  \frac{N^K_{\ell+i+1}(u)}{N^K_{i+1}(u)}du\\
  & + \int_{(s\wedge\theta_K)\log K}^{(t\wedge\theta_K)\log K} \sum_{\ell=-\lfloor 1/2\delta_K\rfloor}^{1/\delta_K-1-\lfloor 1/2\delta_K\rfloor}
  h_{K} p((\ell+i)\delta_K)G(h_{K}\ell)\\
  & \hspace{7cm}\left(\Delta_KM^K_{\ell+i}\big(\frac{u}{\log K}\big)-\Delta_KM^K_{i}\big(\frac{u}{\log K}\big)\right)
  \frac{N^K_{\ell+i+1}(u)}{N^K_{i+1}(u)}du,
\end{aligned}
\end{multline*}}
where $$C_0(K,L) = C+\frac{C\bar G(L)}{\delta_K K^a}.$$   Now, using Corollary \ref{corol:accrM} and~\eqref{eq:borne-N_i/N_j}, on the event $\Omega_K(L)$,
         \begin{align*}
& \int_{(s\wedge\theta_K)\log K}^{(t\wedge\theta_K)\log K} \sum_{\ell=-\lfloor 1/2\delta_K\rfloor}^{1/\delta_K-1-\lfloor 1/2\delta_K\rfloor}
  h_{K} p((\ell+i)\delta_K)G(h_{K}\ell)\\
  & \hspace{7cm}\left|\Delta_KM^K_{\ell+i}\big(\frac{u}{\log K}\big)-\Delta_KM^K_{i}\big(\frac{u}{\log K}\big)\right|
  \frac{N^K_{\ell+i+1}(u)}{N^K_{i+1}(u)}du \\
  & \leq 2\bar p \log K (t-s)\varepsilon_K\sum_{\ell\in\Z} h_K G(k_k\ell)e^{L h_K|\ell|}\leq 2\bar{p}\log K \varepsilon_K \bar G(L)(t-s).
  \end{align*}
Similarly, using~\eqref{eq:borne-martingale-simple}, on the event $\Omega_K(L)$,
\begin{multline*}
    \int_{(s\wedge\theta_K)\log K}^{(t\wedge\theta_K)\log K} \sum_{\ell=-\lfloor 1/2\delta_K\rfloor}^{1/\delta_K-1-\lfloor 1/2\delta_K\rfloor}
  h_{K} p((\ell+i)\delta_K)G(h_{K}\ell)\\
  \times \left|M_{\ell+i+1}\big(\frac{u}{\log K}\big)-M_{i+1}\big(\frac{u}{\log K}\big)\right|
  \frac{N^K_{\ell+i+1}(u)}{N^K_{i+1}(u)}du
  \leq \bar{p}\log K \varepsilon_K \bar G(L)(t-s).
\end{multline*}
To conclude, we use the inequality {$e^x(3-x)\leq e^2$} for all $x\in\mathbb{R}$ to deduce that
\[
\frac{N^K_{\ell+i+1}(u)}{N^K_{i+1}(u)}\left[{3}-\log K\left(\beta_{\ell+i+1}\big(\frac{u}{\log K}\big)-\beta_{i+1}\big(\frac{u}{\log K}\big)\right)\right]\leq {e^2}.
\]
Combining the four previous inequalities, we deduce that
\begin{multline*}g^K_{i}(t)-g^K_{i}(s) \leq C(K,L)(t-s)\\
+
     \bar p \log K \int_{(s\wedge\theta_K)}^{(t\wedge\theta_K)} \sum_{\ell=-\lfloor 1/2\delta_K\rfloor}^{1/\delta_K-1-\lfloor 1/2\delta_K\rfloor}
  h_{K} G(h_{K}\ell) \left[g^K_{\ell+i}(v) - g^K_{i}(v)\right]^+ \, e^{h_{K}L|\ell|}dv,\end{multline*}
where $[x]^+=x\vee 0$ is the positive part of $x$ and
\[
C(K,L)=C+\frac{C\bar G(L)}{\delta_K K^
a}   +\left(2+\frac{\|p\|_{\text{Lip}}}{\underline{p}}\right)\bar{p}\log K \varepsilon_K \bar G(L).
\]
Thus
  $${d g^K_{i}(t)\over dt} \le C(K,L) +  \bar p \log K\sum_{\ell=-\lfloor 1/2\delta_K\rfloor}^{1/\delta_K-1-\lfloor 1/2\delta_K\rfloor}
  h_{K} G(h_{K}\ell) \left[g^K_{\ell+i}(t) - g^K_{i}(t)\right]^+ \, e^{h_{K}L|\ell|}.$$

We will now use the maximum principle for $\omega\in\Omega_K(L)$ fixed.
 Defining $\widetilde g^K_{i}(t) = g^K_{i}(t) -2C(K,L) t$,
  we deduce that for any $t\le \theta_{K}$,
 \begin{equation}
 \label{derive}{d \widetilde g^K_{i}(t)\over dt} < \bar p \log K \sum_{\ell=-\lfloor 1/2\delta_K\rfloor}^{1/\delta_K-1-\lfloor 1/2\delta_K\rfloor}
  h_{K} G(h_{K}\ell) \left[\widetilde g^K_{\ell+i}(t) - \widetilde g^K_{i}(t)\right]^+ \, e^{h_{K}L|\ell|}.
  \end{equation}
  Let us introduce
  $$(i_{K}, t_{K}) = (i_K(\omega),t_K(\omega))=\text{argmax}_{{i\in \{0,\cdots, {1\over \delta_{K}}-1\}, t\in [0, \theta_{K}(\omega)]}} \widetilde g^K_{i}(t)$$
and let us prove that $$t_{K}=0.$$
By contradiction, if we assume that $t_{K}>0$, then  the right term of \eqref{derive} is non-positive for $i=i_{K}$ and then the
left term is negative, contradicting the fact that $\widetilde g^K_{i_K}(t)$ is maximal for $t=t_K$. Hence, we have proved that, almost surely on
the event $\Omega_K(L)$, for all $t\leq \theta_K$ and $0\leq i\leq 1/\delta_K-1$,
\[
g^K_i(t)=\widetilde g^K_i(t)+2C(K,L) t\leq \max_{0\leq j\leq 1/\delta_K-1}\widetilde g^K _j(0)+2 C(K,L) t=\max_{0\leq j\leq 1/\delta_K-1} g^K _j(0)+2C(K,L) t,
\]
so that, by Proposition~\ref{est-varpart},
\begin{align*}
    \Delta_K\beta_i^K(t\wedge\theta_K) & =g_i^K(t) - \frac{\|p\|_{\text{Lip}}}{\underline{p}} A_{i+1}^K(t\wedge\theta_K)+
\Delta_KM_i^K(t\wedge\theta_K) \\ & \leq \max_{0\leq j\leq 1/\delta_K} g^K _j(0)+2 C(K,L) t+\frac{\|p\|_{\text{Lip}}}{\underline{p}}\left(\max_{0\leq j\leq 1/\delta_K-1}\beta^K_i(0)+C_1 t\right)+\varepsilon_K.
\end{align*}
A similar argument applied to $(\beta^K_i(t\wedge\theta_K)-\beta^K_{i-1}(t\wedge\theta_K))/\delta_K$ gives the converse inequality, so we finally obtain that there exists a constant $C$ independent of $K$, $i$, $t$ and $L$ such that, almost surely on the event $\Omega_K(L)$, for all $i\in\{0,\ldots,1/\delta_K-1\}$ and $t\leq T$,
\[
|\Delta_K\beta_i^K(t\wedge\theta_K)|\leq C\left[\max_{0\leq j\leq 1/\delta_K-1} \left(|\Delta_K \beta^K_i(0)|+\beta^K_i(0)\right)+1+T+\bar{G}(L) T\left(\frac{1}{\delta_K K^
a}+\varepsilon_K\log K\right)\right],
\]
Finally, defining $\widetilde \Omega_K$ as the event of probability converging to $1$ where $\max_{0\leq j\leq 1/\delta_K-1} |\Delta_K \beta^K_i(0)|+\beta^K_i(0)\le A+\|\beta_0\|_\infty+1$, where the constant $A$ comes from Assumption~\ref{hypconvergence}, on the event $\Omega_K(L) \cap \widetilde \Omega_K$, we have
\begin{equation*}
|\Delta_K\beta_i^K(t\wedge\theta_K)|\leq  C\left[A+\|\beta_0\|_\infty+2 +T+ \bar G(L) T\left(\frac{1}{\delta_K K^
a}+\varepsilon_K\log K\right)\right].
\end{equation*}

To conclude the proof of Proposition~\ref{prop:Lipschitz-estimate}, we first fix $T>0$, set
\begin{equation}\label{eq:L0}
L_0=C(A+\|\beta_0\|_\infty+3+T)\end{equation}and choose $K$ large enough such that $C\,T\,\bar G(L)\left(\frac{1}{\delta_K K^
a}+\varepsilon_K\log K\right)< 1$.
Then
$$\P(\tau_K(L_0)>T)\ge \P(\Omega_K(L_0)\cap \widetilde \Omega_K)\xrightarrow[K\to\infty]{} 1.$$
Hence Proposition~\ref{prop:Lipschitz-estimate} is proved.
\end{proof}

\subsection{Proof of Theorem~\ref{thm:tight}: tightness of $\tilde \beta^K$}
\label{sec:tight}

Our goal is to prove that the sequence of laws of $(\tilde \beta^K_{t}, t\in[0,T])_{K}$ is tight in ${\cal P}(\mathbb{D}([0,T], \Co(\T,\mathbb{R})))$. We will see in Corollary~
\ref{C-tension} that it is actually $C$-tight.

Let us recall that the random functions $\ \tilde \beta^K\in\mathbb{D}([0,T], \Co(\T,\mathbb{R}))$ is defined in \eqref{lafonction} as follows. For all $x\in\T$, let $i\in\{0,\ldots,1/\delta_K-1\}$ be such that $x\in [i\delta_K,(i+1)\delta_K)$, and set
$$
\widetilde \beta^K_{t}(x):=\widetilde \beta^K(t,x) =\beta_i^K(t)\Big(1-\f {x}{\delta_K}+i\Big)+\beta_{i+1}^K(t)\Big(\f {x}{\delta_K}-i\Big),
$$
where, by convention, $\beta^K_{1/\delta_K}(t)=\beta^K_0(t)$.


Let us recall that the proof of Theorem~\ref{thm:tight} is based on the criterion of Jakubovski \cite{jakubowski} recalled in Section~\ref{sec:main-steps}. Our goal is to prove Conditions~(i) and~(ii) therein.


Let us first prove (i). By Ascoli's Theorem, we know that a compact set $K_{\varepsilon}$ is a set of equi-continuous and equi-bounded functions. By Corollary \ref{lem:upper-bound} and Proposition \ref{prop:Lipschitz-estimate}, we have, on an event of probability converging to $1$ when $K$ tends to infinity,
that, for all $x\in\T$ and all $t\in[0,T]$,
$$\widetilde \beta^K_{t}(x) =(x-i\delta_K)\Delta_K\beta^K_i(t) + \beta_{i}^K(t) \leq  L\delta_K +C(T),$$
so the sequence  $(\widetilde \beta^K_{t}, t\in[0,T])_{K}$  is equi-bounded.
Furthermore, recall that, by~\eqref{def:tau2}, for $x,y\in\T$,
\[\left|\widetilde \beta^K_{t}(x)-\widetilde \beta^K_{t}(y)\right|= \rho(x,y)\sup_{0\leq j\leq 1/\delta_K-1}\left|\Delta_K\beta^K_j(t)\right|
\leq L \rho(x,y).\]
We deduce that the sequence is equi-continuous and~(i) is proved.

Let us now prove (ii), i.e. that for all
$f\in \Co(\T,\mathbb{R})$, the sequence of laws of the  real-valued processes
\begin{align*}
X^K(t)&= \int_\T \widetilde \beta^K(.,x) f(x) dx \\ & =\sum_{i=0}^{1/\delta_K-1} \bigg[\beta^K_{i}(t) \int_{i\delta_{K}}^{(i+1)\delta_{K}} \Big(1+i -{x\over \delta_{K}} \Big)f(x)dx +\beta^K_{i+1}(t)  \int_{i\delta_{K}}^{(i+1)\delta_{K}} \Big({x\over \delta_{K}}-i\Big) f(x)dx \bigg]
\end{align*}
is tight. Recalling that $\beta_{i}^K(t)= A^K_{i}(t) + M^K_{i}(t)$, the process $X^K_f$ is a local semi-martingale with Doob-Meyer decomposition
$X^K_{f} =  A^K_{f}  + M^K_{f} $, where
$$A^K_{f} (t) =  \sum_{i} \bigg[A^K_{i}(t) \int_{i\delta_{K}}^{(i+1)\delta_{K}} (1+i -{x\over \delta_{K}} )f(x)dx +A^K_{i+1}(t)  \int_{i\delta_{K}}^{(i+1)\delta_{K}} ({x\over \delta_{K}}-i) f(x)dx \bigg]
$$
and $M^K_f$ is defined similarly using $M^K_i(t)$ instead of $A^K_i(t)$.

We use Aldous and Rebolledo criteria (see for example Joffe-M\'etivier \cite{joffemetivier}) to prove the tightnes of the sequence $(X^K_f)$.
Let $S$  be a stopping times for the filtration of the underlying Poisson point measures, a.s. in $[0,T]$.
We need to estimate for $\alpha>0$, the quantity
$\mathbb{P}(|A^K_{f}((S+\alpha)\wedge T) -A^K_{f} (S)|>\eta)$ for $\eta>0$.
From~\eqref{def:A}, we deduce
\begin{align*}
&A^K_{f}((S+\alpha)\wedge T) -A^K_{f} (S) =\\ &  \sum_{i} \left\{\Bigg(\int_{i\delta_{K}}^{(i+1)\delta_{K}} (1+i -{x\over \delta_{K}} )f(x)dx\right) \bigg[
 {1\over \log K}\int_{S \log K}^{((S+\alpha)\wedge T)\log K}\Bigg(b(i  \delta_{K})N^K_{i}(s)\log\left(1+{1\over N^K_{i}(s)}\right)\\
 & \hskip 7cm+ d(i  \delta_{K})N^K_{i}(s)\log\left(1-{1\over N^K_{i}(s)}\right)\Bigg) ds
\\
& +{{1\over \log K}}\,   \sum_{\ell=-\lfloor 1/2\delta_K\rfloor}^{1/\delta_K-1-\lfloor 1/2\delta_K\rfloor}h_Kp((\ell+i)\delta_K)G(h_K\ell)\\
& \hskip7cm \times\int_{S \log K}^{((S+\alpha)\wedge T)\log K} N^K_{\ell+i}(s)\log\left(1+{1\over N^K_{i}(s)}\right) ds\Bigg]\\
&+ \left(\int_{i\delta_{K}}^{(i+1)\delta_{K}} ({x\over \delta_{K}}-i) f(x)dx\right) \Bigg[
 {1\over \log K}\int_{S \log K}^{((S+\alpha)\wedge T)\log K}\Bigg(b((i+1)  \delta_{K})N^K_{i+1}(s)\log\left(1+{1\over N^K_{i+1}(s)}\right)\\
 & \hskip7cm + d((i +1) \delta_{K})N^K_{i+1}(s)\log\left(1-{1\over N^K_{i+1}(s)}\right)\Bigg) ds \\
& +{{1\over \log K}} \sum_{\ell=-\lfloor 1/2\delta_K\rfloor}^{1/\delta_K-1-\lfloor 1/2\delta_K\rfloor}h_Kp((\ell+i+1)\delta_K)G(h_K\ell)
\\
& \hskip7cm \times \int_{S \log K}^{((S+\alpha)\wedge T)\log K} N^K_{\ell+i+1}(s)\log\left(1+{1\over N^K_{i+1}(s)}\right) ds\Bigg]\Bigg\}.
\end{align*}
Using \eqref{approxlog3}  and the definition of $\theta_K$, proceeding as in the proof of Proposition~\ref{prop:Lipschitz-estimate} we have
\begin{multline*}
\mathbb{E}(|A^K_{f}((S+\alpha)\wedge\theta_K\wedge T) -A^K_{f} (S\wedge\theta_K)|) \\
\begin{aligned}
& \leq {C\over \log K}\sum_{i=0}^{1/\delta_K-1} \int_{i\delta_{K}}^{(i+1)\delta_{K}} |f(x)|dx\Bigg\{ 2(\bar b+\bar d)\,\alpha\log K\\ 
& \qquad+  \sum_{\ell=-\lfloor 1/2\delta_K\rfloor}^{1/\delta_K-1-\lfloor 1/2\delta_K\rfloor}h_Kp((\ell+i)\delta_K)G(h_K\ell) \mathbb{E}\Big(\int_{(S\wedge\theta_K) \log K}^{((S+\alpha)\wedge\theta_K\wedge T)\log K} {N^K_{\ell+i}(s)\over N^K_{i}(s)}ds\Big)\\
& \qquad+ \sum_{\ell=-\lfloor 1/2\delta_K\rfloor}^{1/\delta_K-1-\lfloor 1/2\delta_K\rfloor}h_Kp((\ell+i+1)\delta_K)G(h_K\ell) \mathbb{E}\Big(\int_{(S\wedge\theta_K) \log K}^{((S+\alpha)\wedge\theta_K\wedge T)\log K} {N^K_{\ell+i+1}(s)\over N^K_{i+1}(s)}ds\Big)\Bigg\}\\
& \leq \alpha C\left[2(\bar b+\bar d)+2\bar p\overline{G}(L) \right]\|f\|_\infty.
\end{aligned}
\end{multline*}
By Proposition \ref{prop:Lipschitz-estimate}, $\theta_K>T$ with probability converging to 1, so we deduce from Markov's inequality that, for all $\varepsilon>0$ and $\eta>0$, there exists $\alpha$ such that,
\[
\limsup_{K\to+\infty}\sup_S\P(|A^K_{f}((S+\alpha)\wedge T) -A^K_{f} (S)|>\eta)\leq\varepsilon,
\]
where the supremum is taken over all stopping times $S\leq T$. This is Aldous criterion for $A^K_{f}(t)$.

It remains to to prove a similar property replacing $A^K_{f}$ by $\langle M^K_{f} \rangle$. This can be done similarly using~\eqref{def:M}. Computations are actually simpler by Lemma~\ref{lem:M}.\hfill$\Box$

\bigskip

Hence we have proved Theorem~\ref{thm:tight}.We prove in the next corollary that the sequence of laws of $(\tilde \beta^K_{t}, t\in[0,T])_{K}$ is actually $C$-tight.

\begin{corollary}
\label{C-tension}
The sequence of laws of  $(\widetilde \beta^K_{t}, t\in[0,T])_{K}$ is C-tight in ${\cal P}(\mathbb{D}([0,T], \Co(\T,\mathbb{R})))$. In addition, for all $T>0$, given a value of $L$ as in Proposition~\ref{prop:Lipschitz-estimate}, for any $\beta$ distributed as a limiting value of the laws of $(\widetilde \beta^K_{t}, t\in[0,T])_{K}$, we have almost surely
\[
\sup_{t\in[0,T]}\sup_{x,y\in\T\ s.t.\ x\neq y}\frac{|\beta(t,x)-\beta(t,y)|}{\rho(x,y)}\leq L.
\]
\end{corollary}

\begin{proof}
Since $|\beta_{i}^K(t)-\beta_{i}^K(t-)| \leq C/\log K$ for any $K$, $i$ and $t$,  we have
$$\lim_{K} \mathbb{P}(\sup_{t\le T} \|\widetilde\beta^K_{t}-\widetilde\beta^K_{t-}\|_\infty >\varepsilon) = 0.$$
Then, we deduce from Proposition 3.26 in Jacod-Shiryaev p.351 \cite{jacod} that, for all $f\in \Co(\T,\R)$, the sequence of laws of $X_f^K$ defined in~\eqref{def-Xf} is $C$-tight. We proceed by contradiction to deduce that $\widetilde{\beta}^K$ is also $C$-tight: if this is not true, there exists an event $\Omega_1$ of positive probability such that, for all $\omega\in\Omega_1$, there exists $t_0(\omega)$, $\alpha(\omega)$ and a ball $B(\omega)\subset \T$ of positive radius such that, for all $x\in B$, \begin{equation}\label{etape1}
|\widetilde{\beta}^K_{t_0}(x)-\widetilde{\beta}^K_{t_0-}(x)|>\alpha.
\end{equation}Therefore, there exists non-random $\alpha>0$ and $\varepsilon>0$ and $i\in\{0,1,\ldots, \lfloor 1/\varepsilon\rfloor-1\}$ and an event $\Omega_2\subset\Omega_1$ of positive probability such that~\eqref{etape1} holds true for all $x\in[i\varepsilon,(i+1)\varepsilon]$ and for this non-random $\alpha$. Now, we define $f_i\in \Co(\T,\R)$ with support in $[i\varepsilon,(i+1)\varepsilon]$ and positive on $(i\varepsilon,(i+1)\varepsilon)$. Then, for all $\omega\in\Omega_2$,
\[
\liminf_{K\to+\infty}|X_{f_i}^K(t_0)-X_{f_i}^K(t_0-)|\geq \alpha \inf_{x\in [(i+1/3)\varepsilon,(i+2/3)\varepsilon]}|f_i(x)|>0.
\]
This is a contradiction with the $C$-tightness of $X^K_{f_i}$.

We now prove the Lipschitz estimate for $\beta$. Using the Skorohod representation theorem, we can construct copies  $\hat{\beta}^K$ of $\tilde{\beta}^K$ and $\hat{\beta}$ of $\beta$ such that $\hat{\beta}^K$ converges (up to a subsequence) almost surely for the $L^\infty$ norm on $[0,T]$ to $\hat{\beta}$. We then define
\[
\hat{\tau}^K=\inf_{i\neq j \in \{0,\ldots,1/\delta_K-1\}}\inf\left\{t\geq 0: \frac{|\hat{\beta}^K(t,i\delta_K)-\hat{\beta}^K(t,j\delta_K)|}{\rho(i\delta_K,j\delta_K)}>L\right\}.
\]
Then $\hat{\tau}_K$ is distributed as $\tau_K$ in~\eqref{def:tau}. It then follows from Propostion~\ref{prop:Lipschitz-estimate} that $\hat{\tau}_K>T$ with probability converging to 1. Hence, for all $x\neq y\in\T$, almost surely
\[
|\hat{\beta}(t,x)-\hat{\beta}(t,y)|=\lim_{K\to+\infty}|\hat{\beta}^K(t,x)-\hat{\beta}^K(t,y)|\leq L \,\rho(x,y).
\]
By continuity of $\hat{\beta}$ we deduce that this property holds, almost surely, for all $x,y\in\T$.
\end{proof}

 \section{Identification of the limit as a viscosity solution of a Hamilton-Jacobi equation}\label{sec:identif}

In the previous section, we obtained the tightness of the laws of $(\widetilde \beta^K_{t}, t\in[0,T])_{K}$ for all $T>0$. Hence, the sequence of laws of $(\widetilde \beta^K_{t}, t\in[0,+\infty))_{K}$ admits at least one limiting value. 
Our aim is now to identify the limiting path as the unique viscosity solution of the Hamilton-Jacobi equation
$$
\frac{\partial}{\partial t} \beta(t,x)= b(x)-d(x) + p(x) \int_{\R} G(h)e^{h\partial_x \beta(t,x)} dh.
$$

Let $\beta$ be distributed as a limiting value of the laws of $(\tilde \beta^K_{t}, t\in[0,+\infty))_{K}$. By Corollary~\ref{C-tension}, $\beta$ belongs to $C([0,+\infty)\times\T,\R)$. In the sequel and with an abuse of notation, we denote again by $(\tilde \beta^K_{t}, t\in[0,+\infty))_{K}$ the subsequence that converges in distribution to $\beta$.

It also follows from Lemma~\ref{lem:M} that
$
\widetilde{A}^K-\widetilde{\beta}^K
$ converges in law, and thus in probability, to $0$. Therefore $(\widetilde{A}^K-\widetilde{\beta}^K,\widetilde{\beta}^K)$ converges in law to $(0, \beta)$ and thus $(\widetilde{A}^K,\widetilde{\beta}^K)$ converges in law to $(\beta, \beta)$.
Using Skorokhod's representation theorem, there exist a new probability space, still denoted (by abuse of notation) by $(\Omega,\mathcal{A},\mathbb{P})$, and random variables still denoted by $\widetilde{A}^K,\widetilde{\beta}^K, \widetilde{M}^K$ and $\beta$ on this space, such that $(\widetilde{A}^K,\widetilde{\beta}^K)$ converges almost surely to $(\beta, \beta)$. Let us denote by $\widetilde{\Omega}_0$ the event where the convergence holds.

We also define for the value $L_0$ defined in Theorem \ref{lem:theta}
\[
\widetilde{\Omega}_K=\left\{\omega\in \Omega \ :\ \sup_{x\in\T}\,\sup_{t\in[0,T]}\,|\widetilde{M}^K(t,x)|\leq \varepsilon'_K,\ \|\widetilde{\beta}^K\|_\text{Lip}\leq L_0 \right\},
\]
where $\varepsilon'_K=\delta_K^{-1/2} K^{-a/2}$ converges to 0 by~\eqref{hyp:deltaK}.
It follows from Lemma~\ref{lem:M} and Proposition~\ref{prop:Lipschitz-estimate} that $\P(\widetilde{\Omega}_K)\to 1$ when $K\to +\infty$. Hence, the set
\[
\Omega_0:=\widetilde{\Omega}_0\cap \limsup_{K\to+\infty} \widetilde{\Omega}_K
\]
has probability 1.

To prove that $\beta$ is a viscosity sub-solution of Equation \eqref{eq:HJ}, we  work 
$\omega$ by $\omega$ in $\Omega_0$.
Let $\omega\in\Omega_0$ and $T>0$ and consider a continuous function $\varphi:[0,T]\times\T$ (depending on $\omega$) such that $\beta(\omega)-\varphi$ attains a strict {global} maximum on {$[0,T]\times\T$} at the point $(\bar t(\omega),\bar x(\omega))$ such that $\bar{t}(\omega)>0$. 
We will prove that
$$
\frac{\partial}{\partial t} \varphi(\bar t,\bar x) \leq {b(\bar x)-d(\bar x)}+ p(x) \int_{\R} G(h)e^{h\partial_x \varphi(\bar t,\bar x)} dh.
$$
Since $\widetilde{A}^K(\omega)$ converges in $L^\infty([0,T]\times\T)$ to $\beta$, there exists for $K$ large enough a local maximum of $\widetilde{A}^K(\omega)-\varphi$ on  $[0,T]\times\T$ at a point $(t_K(\omega),x_K(\omega))$ such that $(t_K(\omega),x_K(\omega))\to (\bar t(\omega),\bar x(\omega))$ as $K\to \infty$. Assume $K$ is large enough so that $t_K(\omega)>0$. 
From now on, we will omit the dependencies with respect to $\omega\in\Omega_0$ to avoid heavy notation.

Defining $i_K\in\{0,\ldots,1/\delta_K-1\}$ such that $i_K\delta_K\leq {x_K}<(i_K+1)\delta_K$, we have
{\begin{align*}
&\frac{\partial}{\partial t} \widetilde A^K(t_K,x_K)  =
(1-\f {x_K}{\delta_K}+i_K)\frac{d}{dt}A_{i_K}^K(t_K)+(\f {x_K}{\delta_K}-i_K) \frac{d}{dt}A_{i_K+1}^K(t_K)\notag\\
&=(1-\f {x_K}{\delta_K}+i_K) N_{i_K}^K(t_K)\Big(b(i_K\delta_K) \log \big(1+\frac{1}{N_{i_K}^K(s)} \big)+d(i_K\delta_K) \log \big(1-\frac{1}{N_{i_K}^K(s)} \big)\Big)\notag\\
&+(\f {x_K}{\delta_K}-i_K) N_{i_K+1}^K(t_K)\Big(b\big((i_K+1)\delta_K\big) \log \big(1+\frac{1}{N_{i_K+1}^K(s)} \big)+d\big((i_K+1)\delta_K\big) \log \big(1-\frac{1}{N_{i_K+1}^K(s)} \big)\Big)\notag\\
&+ (1-\f {x_K}{\delta_K}+i_K)  \sum_{\ell=-\lfloor 1/2\delta_K\rfloor}^{1/\delta_K-1-\lfloor 1/2\delta_K\rfloor}h_Kp((\ell+i_K)\delta_K)G(h_K\ell) N_{\ell+i_K}^K(t_K)\log \big(1+\frac{1}{N_{i_K}^K(t_K)} \big)\notag\\
&+ (\f {x_K}{\delta_K}-i_K)  \sum_{\ell=-\lfloor 1/2\delta_K\rfloor}^{1/\delta_K-1-\lfloor 1/2\delta_K\rfloor}h_Kp((\ell+i_K+1)\delta_K)G(h_K\ell) N_{\ell+i_K+1}^K(t_K)\log \big(1+\frac{1}{N_{i_K+1}^K(t_K)} \big).
\end{align*}}
{Using that forall $x\geq -1/2$,
\begin{equation}b\log(1+x)+d\log(1-x)\le (b-d) x.\label{approxlog2}
\end{equation}and using Lemma \ref{lem:lower-bound}, we deduce that
\begin{align*}
&\frac{\partial}{\partial t} \widetilde A^K(t_K,x_K) \leq (1-\f {x_K}{\delta_K}+i_K) \big(b(i_K\delta_K) -d(i_K\delta_K)\big)(1+\frac{C}{K^a} )
\notag\\
&+(\f {x_K}{\delta_K}-i_K)\big(b((i_K+1)\delta_K) -d((i_K+1)\delta_K)\big)(1+\frac{C}{K^a})\notag\\
&+ (1-\f {x_K}{\delta_K}+i_K)  \sum_{\ell=-\lfloor 1/2\delta_K\rfloor}^{1/\delta_K-1-\lfloor 1/2\delta_K\rfloor}h_Kp((\ell+i_K)\delta_K)G(h_K\ell)e^{\log(K)(\beta_{\ell+i_K}^K(t_K)-\beta_{i_K}^K(t_K))}
(1+\frac{C}{K^a})
\notag\\
&+ (\f {x_K}{\delta_K}-i_K)  \sum_{\ell=-\lfloor 1/2\delta_K\rfloor}^{1/\delta_K-1-\lfloor 1/2\delta_K\rfloor}h_Kp((\ell+i_K+1)\delta_K)G(h_K\ell)e^{\log(K)(\beta_{\ell+i_K+1}^K(t_K)-\beta_{i_K+1}^K(t_K))} (1+\frac{C}{K^a})
\end{align*}
We next use the fact that $\mu$, $b$ and $d$ are $C^1$ functions in $\T$ to obtain, modifying the constant $C$ if necessary,
\begin{align}
&\frac{\partial}{\partial t} \widetilde A^K(t_K,x_K) \leq \big( b(x_K)-d(x_K)+C\delta_K)(1+\frac{C}{K^a} )\notag\\
&+(1-\f {x_K}{\delta_K}+i_K)  \sum_{\ell=-\lfloor 1/2\delta_K\rfloor}^{1/\delta_K-1-\lfloor 1/2\delta_K\rfloor}h_K (p(x_K)+C\delta_K(|\ell|+1)) G(h_K\ell)e^{\log(K)(\beta_{\ell+i_K}^K(t_K)-\beta_{i_K}^K(t_K))} (1+\frac{C}{K^a})
\notag\\
&+ (\f {x_K}{\delta_K}-i_K)  \sum_{\ell=-\lfloor 1/2\delta_K\rfloor}^{1/\delta_K-1-\lfloor 1/2\delta_K\rfloor}h_K (p(x_K)+C\delta_K(|\ell|+1)) G(h_K\ell)e^{\log(K)(\beta_{\ell+i_K+1}^K(t_K)-\beta_{i_K+1}^K(t_K))} (1+\frac{C}{K^a})
 \label{eq:calcul-penible}
\end{align}
}
Since $(t_K,x_K)$ is a maximum point of $\widetilde A^K-\varphi$, we deduce that
\begin{align*}
\lefteqn{\beta_j^K(t_K)-\beta_{i_K}^K(t_K) }\\
&  =\widetilde{\beta}^K(t_K,j\delta_K)-\widetilde{\beta}^K(t_K,x_K)-\left(\beta_{i_K}^K(t_K)-\widetilde{\beta}^K(t_K,x_K)\right) \\ & \leq\varphi(t_K,j\delta_K)-\varphi(t_K,x_K)+\widetilde{M}^K(t_K,j\delta_K)-\widetilde{M}^K(t_K,x_K)-\left(\beta_{i_K}^K(t_K)-\widetilde{\beta}^K(t_K,x_K)\right)
\end{align*}
and similarly for $\beta_{j+1}^K(t_K)-\beta_{i_K+1}^K(t_K)$.
\\
In addition,
$$
\varphi(t_K,j\delta_K)- \varphi(t_K,x_K) \leq (j-i_K)\delta_K\partial_x \varphi(t_K,x_K)+O(|x_K-i_K\delta_K|)+O(|j-i_K|^2\delta_K^2).
$$
Therefore, since $\omega\in\limsup \widetilde{\Omega}_K$, there exists a subsequence in $K$ (still denoted $K$) along which
\begin{align*}
\beta_j^K(t_K)-\beta_{i_K}^K(t_K) & \leq (j-i_K)\delta_K\partial_x \varphi(x_K)+C\left(|x_K-i_K\delta_K|+|j-i_K|^2\delta_K^2+\varepsilon_K\right) \\ & \leq (j-i_K)\delta_K\partial_x \varphi(x_K)+C (|j-i_K|^2\delta_K^2+\delta_K+\varepsilon_K)
\end{align*}
and
\[
\beta_{j+1}^K(t_K)-\beta_{i_K+1}^K(t_K)\leq (j-i_K)\delta_K\partial_x \varphi(t_K,x_K) +C(|j-i_K|^2\delta_K^2+\delta_K+\varepsilon_K).
\]
Combining these inequalities with~\eqref{eq:calcul-penible}, we obtain
\begin{multline}
    \frac{\partial}{\partial t}\varphi(t_K,x_K) \leq\frac{\partial}{\partial t} \widetilde A^K(t_K,x_K)
    \leq \big( b(x_K)-d(x_K)+C\delta_K)(1+\frac{C}{K^a} )\\
    +
    \sum_{\ell=-\lfloor 1/2\delta_K\rfloor}^{1/\delta_K-1-\lfloor 1/2\delta_K\rfloor}h_K (p(x_K)+C\delta_K(|\ell|+1)) G(h_K\ell) 
    e^{h_K\ell\partial_x\varphi(t_K,x_K)+C\frac{(h_K\ell)^2}{\log K}+o(1)} (1+\frac{C}{K^a})
   \label{ineq:xK}
\end{multline}
Since $p(x_K)\to p(\bar{x})$ when $K\to+\infty$, to prove the convergence of the  sum in the right-hand side of \eqref{ineq:xK}, it is sufficient to study the convergence of
\begin{align*}
    S= & \sum_{\ell=-\lfloor 1/2\delta_K\rfloor}^{1/\delta_K-1-\lfloor 1/2\delta_K\rfloor} h_K G\big(h_K\ell \big)e^{h_K\ell\partial_x\varphi(t_K,x_K)+C\frac{(h_K\ell)^2}{\log K}}.
\end{align*}
Recall from Assumption \ref{hypmodele} that $G$ is continuous and that there exists $R>0$ such that $G$ is nonincreasing on $[R,+\infty)$ and nondecreasing on $(-\infty,-R]$. We first notice that
\[
S_0=\sum_{\ell=\lfloor -R/h_K \rfloor}^{\lfloor R/h_K\rfloor } h_K G\big(h_K\ell \big)e^{h_K\ell\partial_x\varphi(t_K,x_K)+C\frac{(h_K\ell)^2}{\log K}}
\]
is a Riemann sum which converges to $\int_{-R}^R G(y) e^{y\partial_x\varphi(\bar t,\bar x)}dy$. Hence we only have to deals with the remainder $S-S_0$. We detail the analysis for
\[
S_+=\sum_{\ell= \lfloor R/h_K\rfloor +1}^{1/\delta_K-1-\lfloor 1/2\delta_K\rfloor}h_K G\big(h_K\ell \big)e^{h_K\ell\partial_x\varphi(t_K,x_K)+C\frac{(h_K\ell)^2}{\log K}}.
\]
A similar computation applies to the lower tail.

For all $\varepsilon>0$, there exists $K_0$ such that for $K\geq K_0$, $|\partial_x\varphi(t_K,x_K)-\partial_x\varphi(\bar{t},\bar{x})| \leq \varepsilon$. Hence, setting $a=\partial_x\varphi(\bar{t},\bar{x})$ and recalling that $G$ is nonincreasing on $[R,+\infty)$, for $K$ large enough,
\begin{align*}
    S_+ & \leq \sum_{\ell= \lfloor R/h_K \rfloor +1}^{1/\delta_K-1-\lfloor 1/2\delta_K\rfloor} h_K G\big(h_K\ell \big)e^{ah_K\ell+C\frac{(h_K\ell)^2}{\log K}+\varepsilon h_K|\ell|} \\
    & \leq \sum_{\ell=\lfloor R/h_K \rfloor +1}^{1/\delta_K-1-\lfloor 1/2\delta_K\rfloor} \int_{h_K(\ell-1)}^{h_K\ell } G\big(h_K\ell \big)e^{ah_K\ell+C\frac{(h_K\ell)^2}{\log K}+\varepsilon h_K|\ell|} \ dy\\
    & \leq \int_{R}^{h_K/2\delta_K} G(y) e^{a y+|a|h_K+C\frac{y^2+h_K(2|y|+h_K))}{\log K}+\varepsilon (|y|+h_K)}\ dy \\
    & \leq \int_{R}^{\log K} G(y) e^{a y+C\frac{y^2+2h_K|y|}{\log K}+\varepsilon |y|}(1+\varepsilon)\ dy.
\end{align*}
Observing that, for $|y|\leq (\log K)^{1/3}$, $\frac{y^2+2h_K|y|}{\log K}\leq (\log K)^{-1/3}+2\delta_K(\log K)^{1/3}\to 0$ when $K\to+\infty$ and that, for $|y|\leq \log K$, $\frac{y^2+2h_K|y|}{\log K}\leq y+2h_K$, we can decompose the domain integration as $[R,\log K]=[R,(\log K)^{1/3})\cup[(\log K)^{1/3},\log K]$ to deduce that, for $K$ large enough,
\begin{align*}
S_+ & \leq (1+2\varepsilon)\int_{R}^{(\log K)^{1/3}}G(y)e^{a y+\varepsilon|y|}dy
+(1+2\epsilon)\int_{(\log K)^{1/3}}^{\log K}G(y) e^{(a+C)y+\varepsilon|y|}dy \\ & \leq(1+2\varepsilon)\int_{R}^{+\infty}G(y)e^{a y+\varepsilon|y|}dy+\frac{1+2\epsilon}{e^{(\log K)^{1/3}}}\int_\mathbb{R} G(y) e^{(a+C+2)|y|}dy.
\end{align*}
Now, by dominated convergence,
\[
\int_{R}^{+\infty}G(y)e^{a y+\varepsilon|y|}dy\xrightarrow[\varepsilon\to 0]{}\int_{R}^{+\infty}G(y)e^{a y}dy.
\]
To conclude, we have proved that
\[
\limsup_{K\to+\infty} S\leq \int_\mathbb{R}G(y)e^{ y\partial_x\varphi(\bar{t},\bar{x})}dy.
\]
Therefore,
\[
\frac{\partial}{\partial t}\varphi(\bar{t},\bar{x})\leq b(\bar x)-d(\bar x)+p(\bar x)\int_{\R} G(h) e^{h\partial_x\varphi(\bar{t},\bar{x})}dh.
\]
We conclude that $\beta$ is a viscosity sub-solution of~\eqref{eq:HJ} in $(0,T]\times \T$.

Following similar arguments, we can prove that $\beta$ is a viscosity super-solution, and hence a viscosity solution of~\eqref{eq:HJ} in $(0,T]\times \T$. The result then follows from uniqueness of a Lipschitz viscosity solution of~\eqref{eq:HJ}~\cite{GB:94}.




\begin{thebibliography}{10}

\bibitem{GB:94}
G.~Barles.
\newblock {\em  Solutions de viscosit\'e des \'equations de
  {H}amilton-{J}acobi.}
\newblock Springer-Verlag Berlin Heidelberg, 1994.


\bibitem{Barles2009}
G.~Barles, S.~Mirrahimi, and B.~Perthame.
\newblock Concentration in lotka-volterra parabolic or integral equations: a
  general convergence result.
\newblock {\em Methods and Applications of Analysis}, 16(3):321--340, 2009.

\bibitem{billiardcolletferrieremeleardtran}
S. Billiard, P. Collet, R. Ferri\`ere, S. M\'el\'eard and V.C. Tran.
\newblock Stochastic dynamics for adaptation and evolution of microorganisms
\newblock European Congress of Mathematics Berlin 2016, V. Mehrmann and M. Skutella eds., 525-550, EMS (2018).


\bibitem{blathpaultobias}
 J. Blath, T. Paul and A. T\'obi\'as.
\newblock A Stochastic Adaptive Dynamics Model for Bacterial Populations with Mutation, Dormancy and Transfer
\newblock arXiv:2105.09228 (2021).


\bibitem{Bovier2019}
A.~Bovier,  L.~Coquille, C.~Smadi.
\newblock Crossing a fitness valley as a metastable transition in a stochastic population model.
\newblock {\em Ann. Appl. Probab.} 29(6), 3541--3589 (2019).

\bibitem{calvez20}
V.~ Calvez, S.~Iglesias Figueroa; H.~Hivert, S.~ Méléard, A.~ Melnykova, S.~Nordmann.
\newblock Horizontal gene transfer: numerical comparison between stochastic and deterministic approaches.
\newblock {\em ESAIM}, Proc. Surv. 67, 135--160 (2020).

\bibitem{champagnat06}
N.~Champagnat.
\newblock A microscopic interpretation for adaptative dynamics trait
  substitution sequence models.
\newblock {\em Stochastic Processes and their Applications}, 116:1127--1160,
  2006.

  \bibitem{Champagnat-Henry-2019}
  N.~Champagnat, B.~Henry.
  \newblock A probabilistic approach to Dirac concentration in nonlocal models of adaptation with several resources.
\newblock {\em The Annals of Applied Probability}, 29(4):2175--2216, 2019.



\bibitem{champagnatmeleard2011}
N.~Champagnat and S.~M\'{e}l\'{e}ard.
\newblock Polymorphic evolution sequence and evolutionary branching.
\newblock {\em Probability Theory and Related Fields}, 151(1-2):45--94, 2011.


\bibitem{champagnatmeleardtran2021}
N.~Champagnat, S.~M\'{e}l\'{e}ard and V.C.~Tran.
\newblock Stochastic analysis of emergence of evolutionary cyclic behavior in population dynamics with transfer.
\newblock {\em Ann. Appl. Probab.}, 31(4), 1820--1867, 2021.


\bibitem{coquille2021}
 L.~Coquille, A.~Kraut,  C.~Smadi.
\newblock Stochastic individual-based models with power law mutation rate on a general finite trait space.
\newblock {\em Electron. J. Probab.} 26(123), 37 p., 2021.


\bibitem{dieckmannlaw}
U.~Dieckmann and R.~Law.
\newblock The dynamical theory of coevolution: a derivation from stochastic
  ecological processes.
\newblock {\em Journal of Mathematical Biology}, 34:579--612, 1996.

\bibitem{diekmannjabinmischlerperthame}
O.~Diekmann, P.-E. Jabin, S.~Mischler, and B.~Perthame.
\newblock The dynamics of adaptation: an illuminating example and a
  {H}amilton-{J}acobi approach.
\newblock {\em Theoretical Population Biology}, 67, 257--271, 2005.

\bibitem{durrettmayberry}
R.~Durrett and J.~Mayberry.
\newblock Travelling waves of selective sweeps.
\newblock {\em Annals of Applied Probability}, 21(2), 699--744, 2011.


\bibitem{foriengarnierpatout}
R. Forien, J. Garnier and F. Patout.
\newblock Ancestral lineages in mutation-selection equilibria with moving optimum.
\newblock arXiv:2011.05192 (2020).


\bibitem{jabin2012}
P.-E.~Jabin.
\newblock Small populations corrections for selection-mutation models.
\newblock {\em Netw. Heterog. Media}, 7(4), 805--836, 2012.

\bibitem{jacod}
J.~Jacod and A.N. Shiryaev.
\newblock {\em Limit Theorems for Stochastic Processes}.
\newblock Springer-Verlag, Berlin, 1987.

\bibitem{jakubowski}
A.~Jakubowski.
\newblock On the {S}korokhod topology.
\newblock {\em Annales de l'Institut Henri Poincar\'{e}}, 22(3), 263--285, 1986.

\bibitem{joffemetivier}
A.~Joffe and M.~M\'{e}tivier.
\newblock Weak convergence of sequences of semimartingales with applications to
  multitype branching processes.
\newblock {\em Advances in Applied Probability}, 18, 20--65, 1986.

\bibitem{keelingpalmer2008}
P. Keeling and J. Palmer.
\newblock Horizontal gene transfer in eukaryotic evolution.
\newblock {\em Nat. Rev. Genet.} 9:605--618 (2008).


\bibitem{LMP}  A. Lorz, S. Mirrahimi and B Perthame.
\newblock   Dirac mass dynamics in a multidimensional nonlocal parabolic equation.  \newblock {\em Communications in Partial Differential Equations},  36, 1071--1098 (2011).

\bibitem{metzgeritzmeszenajacobsheerwaarden}
J.A.J. Metz, S.A.H. Geritz, G.~Mesz\'{e}na, F.A.J. Jacobs, and J.S.~Van
  Heerwaarden.
\newblock Adaptative dynamics, a geometrical study of the consequences of
  nearly faithful reproduction.
\newblock {\em S.J. Van Strien \& S.M. Verduyn Lunel (ed.), Stochastic and
  Spatial Structures of Dynamical Systems}, 45, 183--231, 1996.

\bibitem{Mirrahimi2012}
S.~Mirrahimi, G.~Barles, B.~Perthame,  P.E.~Souganidis.
\newblock A singular Hamilton-Jacobi equation modeling the tail problem.
\newblock {\em SIAM J. Math. Anal.}, 44 (6), 4297--4319, 2012.

\bibitem{ochmanlawrencegroisman}
H. Ochman and J. Lawrence and R. Groisman.
\newblock Lateral gene transfer and the nature of bacterial innovation.
\newblock {\em Nature}, 405, 299--304 (2000).



\bibitem{BP.GB:08}   B. Perthame and G. Barles.
\newblock Dirac concentrations in Lotka-Volterra parabolic PDEs.
\newblock {\em Indiana Univ. Math. J.} 57,  3275--3301 (2008).


\bibitem{perthame2010}
B.~Perthame,
M.~Gauduchon.
\newblock Survival thresholds and mortality rates in adaptive dynamics: conciliating deterministic and stochastic simulations.
\newblock {\em Math. Med. Biol.} 27(3), 195--210 (2010).

\bibitem{stewartlevin}
F.M. Stewart and B.R. Levin.
\newblock The population biology of bacterial plasmids: A priori conditions
for the existence of conjugationally transmitted factors.
\newblock {\em Genetics}, 87:209--228  (1977).

\bibitem{WaxmanGavrilets2005}
D. Waxman and S. Gavrilets.
\newblock 20 Questions on Adaptive Dynamics.
\newblock {\em Journal of Evolutionary Biology} 18, 1139-1154 (2005).
\end{thebibliography}
\bigskip
\noindent \textbf{Acknowledgements}:
This work has been supported by the Chair ``Modélisation Mathématique et Biodiversité" of Veolia Environnement-Ecole Polytechnique-Museum National d’Histoire Naturelle-Fondation X. S.M. was partially supported by the ANR project DEEV ANR-20-CE40-0011-01. V.C.T. also acknowledges support from Labex Bézout (ANR-10-LABX-58).

{\footnotesize
\providecommand{\noopsort}[1]{}\providecommand{\noopsort}[1]{}\providecommand{\noopsort}[1]{}\providecommand{\noopsort}[1]{}

}
\end{document}